\documentclass[english,10pt]{amsart}
\usepackage[english]{babel}

\usepackage[matrix,arrow]{xy}
\xyoption{all}
\usepackage{amscd,amssymb,amsfonts,amsmath}

\usepackage{graphics}
\usepackage{epsfig}
\usepackage{mathrsfs}

\newcommand\mypagesizel{
\textwidth= 6.5in
\textheight=9in
\voffset-.55in
\hoffset -0.75in
\marginparwidth=56pt
}

\newcommand{\Pic}{\textup{Pic}}

\newcommand{\Supp}{\textup{Supp}}

\newcommand{\Chow}{\textup{Chow}}

\newcommand{\p}[0]{{\mathbb P}}

\newcommand{\codim}{\textup{codim}}

\renewcommand{\phi}{\varphi}
\newcommand{\into}{\hookrightarrow}
\newcommand{\map}{\dashrightarrow}

\renewcommand{\le}{\leqslant}
\renewcommand{\ge}{\geqslant}

\newcommand{\bN}{\textbf{N}}

\newcommand{\bP}{\mathbb{P}}
\newcommand{\bQ}{\mathbb{Q}}
\newcommand{\bZ}{\mathbb{Z}}

\newcommand{\sE}{\mathscr{E}}
\newcommand{\sF}{\mathscr{F}}
\newcommand{\sG}{\mathscr{G}}
\newcommand{\sH}{\mathscr{H}}
\newcommand{\sI}{\mathscr{I}}

\newcommand{\sK}{\mathscr{K}}
\newcommand{\sL}{\mathscr{L}}
\newcommand{\sM}{\mathscr{M}}

\newcommand{\sO}{\mathscr{O}}

\newcommand{\sQ}{\mathscr{Q}}
\newcommand{\sR}{\mathscr{R}}

\newcommand{\sV}{\mathscr{V}}

\mypagesizel

\newtheorem{thm}{Theorem}[section]

\newtheorem{lemma}[thm]{Lemma}
\newtheorem{cor}[thm]{Corollary}

\newtheorem{prop}[thm]{Proposition}

\newtheorem*{thm*}{Theorem}
\theoremstyle{definition}
\newtheorem{defn}[thm]{Definition}

\newtheorem{say}[thm]{}
\newtheorem{exmp}[thm]{Example}

\newtheorem{notation}[thm]{Notation}

\newtheorem{defn-thm}[thm]{Definition-Theorem} 
\newtheorem{defn-lemma}[thm]{Definition-Lemma}

\newtheorem{rem}[thm]{Remark}

\theoremstyle{remark}
\newtheorem*{claim}{Claim}

\newtheorem*{not-and-def}{Notation and definitions}

\numberwithin{equation}{section}

\begin{document}

\title[On codimension 1 del Pezzo foliations]{On codimension 1 del Pezzo foliations on varieties with mild singularities}

\author{Carolina \textsc{Araujo}} 

\address{\noindent Carolina Araujo: IMPA, Estrada Dona Castorina 110, Rio de
  Janeiro, 22460-320, Brazil} 

\email{caraujo@impa.br}

\author{St\'ephane \textsc{Druel}}

\address{St\'ephane Druel: Institut Fourier, UMR 5582 du
  CNRS, Universit\'e Grenoble 1, BP 74, 38402 Saint Martin
  d'H\`eres, France} 

\email{druel@ujf-grenoble.fr}

\thanks{The first named author was partially supported by CNPq and Faperj Research 
  Fellowships}

\thanks{The second named author was partially supported by the CLASS project of the 
A.N.R}


\begin{abstract}
In this paper we extend to the singular setting the theory of Fano foliations developed in our previous paper
\cite{fano_foliation}. 
A $\bQ$-Fano foliation on a complex projective variety $X$ is a 
foliation $\sF\subsetneq T_X$ whose anti-canonical class $-K_{\sF}$ is an ample $\bQ$-Cartier divisor.
In the spirit of Kobayashi-Ochiai Theorem, we prove that under some conditions the index $i_{\sF}$ of a $\bQ$-Fano foliation 
is bounded by the rank $r$ of $\sF$, and classify the cases in which $i_{\sF}=r$. 
Next we consider $\bQ$-Fano foliations $\sF$ for which $i_{\sF}=r-1$.
These are called del Pezzo foliations. 
We classify codimension 1 del Pezzo foliations on mildly singular varieties.
\end{abstract}

\maketitle

\tableofcontents


\section{Introduction}

The study of Fano manifolds is  a classical theme in algebraic geometry.
Much is known about these varieties, as they appear naturally in several different contexts. 
With the development of the Minimal Model Program  in dimension $\ge 3$, 
it became clear that singularities are unavoidable in the birational classification 
of higher dimensional complex projective varieties. 
This led to the development of a powerful theory of  \emph{singularities of pairs}
(see Definition~\ref{Singularities of pairs} for basic notions, such as klt singularities).
As a consequence of the Minimal Model Program, every smooth complex projective varieties can be build up from
\emph{mildly singular} projective varieties $X$, for which  $K_X$ is $\bQ$-Cartier and satisfies one of the following:
$K_X<0$, $K_X\equiv 0$ or $K_X>0$. 
This makes imperative to extend the theory of Fano manifolds to the singular setting.
A $\bQ$-Fano variety is a normal projective variety $X$ such that $-K_X$ is  an ample $\bQ$-Cartier divisor.
Its index is  the largest positive rational number $i_X$ such that
$-K_{X} \sim_\bQ i_X H$ for a Cartier divisor $H$ on $X$.
When $X$ is an $n$-dimensional smooth Fano variety, a classical result of 
Kobayashi-Ochiai asserts that $i_X\le n+1$.
Moreover, equality holds if and only if $X\simeq\bP^n$, and $i_X= n$ if and only if 
$X$ is a  quadric hypersurface in $\bP^{n+1}$.
We generalize this to the singular setting (see also \cite{sano96} for related results).

\begin{thm}\label{thm:ko}
Let $\Delta$ be an effective $\bQ$-divisor on a normal projective variety $X$ of dimension $n\ge 1$.
Suppose that $K_X+\Delta$ is $\bQ$-Cartier, and $-(K_X+\Delta)\sim_\bQ i c_1(\sL)$ for 
an ample line bundle $\sL$ on $X$ and $i\in\bQ$.  

\begin{enumerate}
\item If $i > n$, then $n < i \le n+1$,
$(X,\sL) \simeq \big(\bP^n,\sO_{\bP^n}(1)\big)$ and
$\deg(\Delta)=n+1-i$.
\item If $i = n$, then either 
$(X,\sL)\simeq 
\big(\bP^n,\sO_{\bP^{n}}(1)\big)$ and 
$\deg(\Delta)=1$,  
or $\Delta=0$ and
$(X,\sL)\simeq (Q_{n},\sO_{Q_{n}}(1))$, where 
$Q_{n}$ is a (possibly singular) quadric hypersurface in $\bP^{n+1}$.
\end{enumerate}
\end{thm}

The Fano condition and the notion of index 
naturally find counterparts in the theory of holomorphic foliations.
Given a holomorphic foliation $\sF\subsetneq T_X$ of  generic rank $r$ on a normal complex projective variety $X$, 
we denote by $\det(\sF)$  the reflexive sheaf $(\wedge^r \sF)^{**}$,
and by $-K_{\sF}$ any Weil divisor on $X$ such that $\sO_X(-K_{\sF})\simeq \det(\sF)$. 
We say that $\sF$ is a \emph{$\bQ$-Fano foliation}  if $-K_{\sF}$ is an ample $\bQ$-Cartier divisor on $X$.
In this case, the \emph{index  of  $\sF$} is the largest positive rational number $i_{\sF}$ such that
$-K_{\sF} \sim_\bQ i_{\sF} H$ for a Cartier divisor $H$ on $X$.

In our previous paper \cite{fano_foliation}, inspired by the theory of Fano manifolds and 
the Minimal Model Program, 
we initiated the study of Fano foliations on complex projective manifolds.
By \cite[Theorem 1.1]{adk08}, the index $i_{\sF}$ of 
a Fano foliation $\sF\subsetneq T_X$ of rank $r$ on a
complex projective manifold $X$ satisfies  $i_{\sF}\le r$. 
Moreover, equality holds only if $X\simeq \p^n$.
This may be viewed as a version for foliations of the classical Kobayashi-Ochiai
Theorem.
Fano foliations on $\bP^n$ with index equal to the rank
were classified in  \cite[Th\'eor\`eme 3.8]{cerveau_deserti}.
They are induced by  linear projections $\p^n \dashrightarrow \p^{n-r}$, and are classically known as 
\emph{degree $0$ foliations on $\bP^n$}.
In \cite{fano_foliation}, we defined \textit{del Pezzo foliations} on complex projective manifolds $X$
as Fano foliations $\sF\subsetneq T_X$ of rank $r\ge 2$ and index $i_{\sF} = r-1$.
Del Pezzo foliations $\sF$  on ${\p^n}$  (classically  known as \emph{degree $1$ foliations} on ${\p^n}$) 
were classified in 
\cite[Theorem 6.2]{lpt3fold}:
\begin{enumerate}
\item Either $\sF$ is induced by a dominant rational  map $\p^n\dashrightarrow \p(1^{n-r},2)$,
defined by $n-r$ linear forms and one quadratic form, or 
\item $\sF$ is the linear pullback of a foliation on $\p^{n-r+1}$ induced by a global 
holomorphic vector field.
\end{enumerate}
The generic del Pezzo foliation of type (2) above does not have algebraic leaves.
On the other hand, we showed in \cite[Theorem 1.1]{fano_foliation} that if $X\not\simeq \p^n$, then
del Pezzo foliations on $X$ have algebraic and rationally connected general leaves.
We then worked toward the classification of del Pezzo foliations.

In this paper we address $\bQ$-Fano foliations on possibly \emph{singular} varieties. 
Our main goal is to extend the theory developed in \cite{fano_foliation} to the singular setting. 
There is at least one strong reason for studying this more general case.
The birational classification of rank $1$ foliations on smooth surfaces
obtained in \cite{brunella} was very much modeled after the birational classification 
of smooth surfaces. 
If one hopes to apply ideas from the Minimal Model Program to approach the problem of birational classification
of holomorphic foliations in higher dimensions, 
then one may be led to study foliations on mildly singular varieties as well.

We first prove a version for foliations on midly singular varieties of the Kobayashi-Ochiai
Theorem. See Definition~\ref{Normal generalized cones} for the notion of normal generalized cone,
and \cite{wahl83}, \cite{andreatta_wisniewski}, \cite{druel04}, \cite{adk08}, and 
\cite{hoeringKOfoliations} for related results.

\begin{thm}\label{thm:codimension_r_Fano_index__equal_rank_picard_number_1}
Let $X$ be an $n$-dimensional $\bQ$-factorial klt projective variety with $n\ge 2$, and
$\sF\subsetneq T_X$ a $\bQ$-Fano foliation of rank $r$.
Suppose that either $X$ has Picard number $\rho(X)=1$, or $\sF$ has rank $r=n-1$. 
Then 
\begin{enumerate}
	\item $i_{\sF}\le r$.
	\item If $i_{\sF}= r$, then $X$ is a normal generalized cone over a polarized  variety $(Z,\sM)$ 
		with vertex $P\simeq\bP^{r-1}$, and $\sF$ is induced by the natural map $X\dashrightarrow Z$. 
		Moreover, $Z$ is $\bQ$-factorial and klt, $\rho(Z)=1$ and $i_Z=k(i_X-r)>0$ for some positive integer $k$.
\end{enumerate}
\end{thm}

Next we consider del Pezzo foliations, i.e., 
$\bQ$-Fano foliations of rank $r\ge 2$ and index $i_{\sF} = r-1$.
In the smooth case, 
del Pezzo foliations of codimension $1$ on Fano manifolds with Picard 
number $1$ were classified in \cite[Proposition 3.7]{lpt3fold}:
they are either degree $1$ foliations on $\bP^n$, or come from 
pencils of hyperplane sections of quadric hypersurfaces. 
We extend this classification to mildly singular varieties, without restriction on
the Picard number.

\begin{thm}\label{thm:classification_del_pezzo}
Let $\sF\subsetneq T_X$ be a codimension $1$ del Pezzo foliation on 
a factorial  klt projective variety $X$ of dimension $n\ge 3$.
\begin{enumerate}
\item Suppose that $\rho(X)=1$. Then $X$ is a $\bQ$-Fano variety and one of the following holds.
\begin{enumerate}
\item $\sF$ is a degree $1$ foliation on $X\simeq\bP^n$.
\item $X$ is isomorphic to a (possibly singular) quadric hypersurface $Q_n\subset \bP^{n+1}$, and $\sF$ is a pencil of hyperplane sections of $Q_n$.
\item $X$ is a normal generalized cone over a projective normal surface $Z$ with $\rho(Z)=1$, and 
$\sF$ is the pullback by $X\dashrightarrow Z$
of a foliation induced by a nonzero global section of $T_Z$.
Moreover, $Z$ is factorial with klt singularities, $i_X>n-2$ and
$i_Z=k(i_X -n+2)$ for some positive integer $k$.
\end{enumerate}
\item Suppose that $\rho(X) \ge 2$.  
Then  there is an exact sequence of vector bundles 
$0\to \sK\to \sE\to \sV \to 0$ on $\bP^1$ such that $X\simeq \p_{\p^1}(\sE)$, and $\sF$ is the pullback via the relative linear projection $X\map Z=\p_{\p^1}(\sK)$ of a foliation 
on $Z$  induced by a nonzero global section of $T_Z\otimes q^*\det(\sV)^*$.
Here $q:Z\to \p^1$ denotes the natural projection.  
Moreover, one of the following holds.
\begin{enumerate}
\item 
$(\sE,\sK)\simeq \big(\sO_{\bP^1}(2)\oplus \sO_{\bP^1}(a)^{\oplus 2},\sO_{\bP^1}(a)^{\oplus 2}\big)$
for some positive integer $a$.
\item 
$(\sE,\sK)\simeq \big(\sO_{\bP^1}(1)^{\oplus 2}\oplus \sO_{\bP^1}(a)^{\oplus 2},\sO_{\bP^1}(a)^{\oplus 2}\big)$
for some positive integer $a$.
\item 
$(\sE,\sK)\simeq\big(\sO_{\bP^1}(1)\oplus \sO_{\bP^1}(a)\oplus \sO_{\bP^1}(b),\sO_{\bP^1}(a)\oplus \sO_{\bP^1}(b)\big)$ for distinct 
positive integers $a$ and $b$.
\end{enumerate}
\end{enumerate}
\end{thm}

One of the key ingredients of our proofs of Theorem~\ref{thm:codimension_r_Fano_index__equal_rank_picard_number_1}
and \ref{thm:classification_del_pezzo} is the following result, which illustrates the strong restrictions 
imposed by the existence of a $\bQ$-Fano foliation on a variety $X$.

\begin{thm}\label{thm:main_foliation} \label{thm:conormal_fano_fols}
Let $X$ be a klt projective variety, and $\sF\subsetneq T_X$ a $\bQ$-Fano foliation.
Then $K_X-K_\sF$ is not pseudo-effective.
\end{thm}

When $\rho(X)=1$, Theorem~\ref{thm:main_foliation} allows us to bound the index of $X$ from below: $i_X > i_\sF$.
When $\rho(X)>1$, it bounds from below the length of extremal rays of $X$.
When $\sF$ is a codimension $1$ del Pezzo foliation, these bounds are good enough for us to recover $X$.
This is not the case for higher codimension or smaller index $i_\sF$.


\medskip

The paper is organized as follows. 

In Section~\ref{section:Q-Fano}, we review basic definitions of singularities of pairs,
gather some results about $\bQ$-Fano varieties and singular del Pezzo varieties, and prove
a vanishing result (Theorem~\ref{thm:main}) that has useful applications to the theory of $\bQ$-Fano foliations.
In particular, it allows us to prove Theorem~\ref{thm:conormal_fano_fols}. 

In Section~\ref{section:foliations}, we develop the basic theory of 
foliations and Pfaff fields on mildly singular varieties, and prove Theorem~\ref{thm:conormal_fano_fols}.

In Section~\ref{section:ko}, in the spirit of Kobayashi-Ochiai Theorem,
we investigate upper bounds for the index of $\bQ$-Fano foliations.
In particular, 
we prove Theorem~\ref{thm:codimension_r_Fano_index__equal_rank_picard_number_1}. 

Section~\ref{section:dp} is devoted to del Pezzo foliations, and
contains the proof of Theorem~\ref{thm:classification_del_pezzo}.

In Section~\ref{section:regular}, we discuss singularities of $\bQ$-Fano foliations. 
In particular, we prove the following result, which says that  
that codimension $1$ $\bQ$-Fano foliations on midly singular varieties must be singular in codimension $2$.

\begin{thm}\label{thm:regular_fano_fols}
Let $\sF\subsetneq T_X$ be a codimension 1 
foliation on a klt projective variety $X$. 
Suppose that $X$ and $\sF$ are both regular in codimension 2.
Then $-K_\sF$ is not ample.
\end{thm}

\

\noindent {\bf Notation and conventions.}
We always work over the field ${\mathbb C}$ of complex numbers. 
Varieties are always assumed to be irreducible and reduced.
We denote by $\textup{Sing}(X)$ the singular locus of a variety $X$.

Let $X$ be a normal variety.
Given a Weil divisor $D$ on $X$, 
we denote by $\sO_X(D)$
the rank $1$ reflexive sheaf defined by: 
$H^0(U,\sO_X(D)_{|U})=\big\{f\in k(X)\setminus \{0\}\ \big| \ \textup{div}(f)+D_{|U}\ge 0\big\}\cup \{0\}$ for any open subset $U\subset X$.
Given a rank $1$ reflexive sheaf  $\sF$  on  $X$, then, by abuse of notation, we denote by
$c_1(\sF)$ any Weil divisor $D$ on $X$ such that $\sO_X(D)\simeq \sF$.

Given two $\bQ$-Weil divisors $D$ and $D'$ on a normal variety $X$, we say that $D \le D'$ if
$D'-D$ is effective.


Let $X$ be a normal variety, and $D$ a $\bQ$-Weil divisor on $X$.
Let $\pi : Y \to X $ be a finite morphism such that, for any prime
divisor $P\subset Y$, $X$ is smooth at the generic point of $\pi(P)$. 
Then the pullback of $D$ to $Y$ is well defined, and we denote it either by $\pi^*D$ or by $D_{|Y}$.

If $\sE$ is a locally free sheaf of $\sO_X$-modules on a variety $X$, we denote by $\p_X(\sE)$ the Grothendieck projectivization $\textup{Proj}_X(\textup{Sym}(\sE))$.

Given a sheaf $\sF$ of $\sO_X$-modules on a variety $X$, we denote by $\sF^{*}$ the sheaf 
$\sH\hspace{-0.1cm}\textit{om}_{\sO_X}(\sF,\sO_X)$.
If $r$ is the generic rank of $\sF$, then we denote by $\det (\sF)$ the sheaf $(\wedge^r \sF)^{**}$.
For $m\in\bN$, we denote by  $\sF^{[m]}$ the sheaf $(\sF^{\otimes m})^{**}$.
If $\sG$ is another sheaf of $\sO_X$-modules on $X$, then we denote by  $\sF[\otimes]\sG$ the sheaf $(\sF\otimes\sG)^{**}$.
If $\pi:Y \to X$ is a morphism of varieties, then we write $\pi^{[*]}\sF$ for $(\pi^*\sF)^{**}$.

If $X$ is a normal variety and $X\to Y$ is any morphism, we denote by
$T_{X/Y}$ the sheaf $(\Omega_{X/Y}^1)^*$. In particular, $T_X=(\Omega_{X}^1)^*$.
If $D$ is a reduced divisor on $X$ and $q\in\bN$, then we denote by 
$\Omega_X^{[q]}$ the sheaf $(\Omega_X^{q})^{**}$, and by 
$\Omega_X^{[q]}(\textup{log }D)$ the sheaf of reflexive differential $q$-forms with logarithmic poles
along $D$. 

\

\noindent {\bf Acknowledgements.}
Much of this work was developed during the authors' visits to IMPA and Institut Fourier.
We would like to thank both institutions for their support and hospitality. 
We would also like to thank the referee for helpful comments.


\section{$\bQ$-Fano varieties, del Pezzo varieties and a vanishing result}\label{section:Q-Fano}

In this section we gather some results about $\bQ$-Fano varieties and singular del Pezzo varieties. 
At the end of the section we prove a vanishing result that has useful applications to the theory of $\bQ$-Fano foliations.
We start by recalling some definitions of singularities of pairs, developed in the context of the Minimal Model Program.

\subsection{Singularities of pairs}

\begin{defn}[See {\cite[section 2.3]{kollar_mori}}]\label{Singularities of pairs}
Let $X$ be a normal projective variety, and
$\Delta=\sum a_i\Delta_i$ an effective $\bQ$-divisor on $X$, i.e., $\Delta$ is  a nonnegative $\bQ$-linear combination 
of distinct prime Weil divisors $\Delta_i$'s on $X$. 
Suppose that $K_X+\Delta$ is $\bQ$-Cartier, i.e.,  some nonzero multiple of it is a Cartier divisor. 

Let $f:\tilde X\to X$ be a log resolution of the pair $(X,\Delta)$. 
This means that $\tilde X$ is a smooth projective
variety, $f$ is a birational projective morphism whose exceptional locus is the union of prime divisors $E_i$'s, 
and the divisor $\sum E_i+\tilde{\Delta}$ has simple normal crossing 
support, where $\tilde{\Delta}$ denotes the strict transform of $\Delta$ in $X$.  
There are uniquely defined rational numbers $a(E_i,X,\Delta)$'s such that
$$
K_{\tilde X}+\tilde{\Delta} = f^*(K_X+\Delta)+\sum_{E_i}a(E_i,X,\Delta)E_i.
$$
The $a(E_i,X,\Delta)$'s do not depend on the log resolution $f$,
but only on the valuations associated to the $E_i$'s. 

We say that $(X,\Delta)$ is \emph{canonical} if  all $a_i\le 1$, and, for some  log resolution 
$f:\tilde X\to X$ of $(X,\Delta)$, $a(E_i,X,\Delta)\ge 0$ 
for every $f$-exceptional prime divisor $E_i$.
We say that $(X,\Delta)$ is \emph{log terminal} (or \emph{klt})  if  all $a_i<1$, and, for some  log resolution 
$f:\tilde X\to X$ of $(X,\Delta)$, $a(E_i,X,\Delta)>-1$ 
for every $f$-exceptional prime divisor $E_i$.
If these conditions hold for some log resolution of $(X,\Delta)$, then they hold for every  
log resolution of $(X,\Delta)$.

We say that $X$ is canonical (respectively  klt)  if so is $(X,0)$.
Note that if $X$ is Gorenstein, i.e., $K_X$ is Cartier, then $X$ is canonical if and only if it is  klt.

\end{defn}

We will need the following observation.

\begin{lemma}\label{lemma:proj_klt}
Let $r$ be a positive integer, and $\Delta$ an effective $\bQ$-divisor on $\bP^r$ of degree $\le 1$. 
Then $(\bP^r,\Delta)$ is klt unless $\Delta=H$, where $H$ is a hyperplane. 
\end{lemma}

\begin{proof}
When $r=1$, the result is clear. We assume from now on that $r\ge 2$. 

Write $\Delta=\sum_{1\le i\le k}d_i \Delta_i$, where $\Delta_i\subset \bP^r$ is an integral hypersurface, 
$\Delta_i\neq\Delta_j$ if $i\neq j$ and the $d_i$'s are positive rational numbers. 
By assumption, $\deg(\Delta)=\sum_{1\le i\le k}d_i\,\deg(\Delta_i)\le 1$.
Suppose that either $k\ge 2$, or $k=1$ and $\deg(\Delta_1)\ge 2$. Let $p\in\bP^r$ be a point. We shall show that $(X,\Delta)$ is klt at $p$.

Observe that
$$\textup{mult}_p\,\Delta=\sum_{1\le i\le k}d_i\,\textup{mult}_p\,\Delta_i \le \sum_{1\le i\le k}d_i\,\deg(\Delta_i)=\deg(\Delta)\le 1.$$

If $\textup{mult}_p\,\Delta<1$, then $(\bP^r,\Delta)$ is klt at $p$ by \cite[Proposition 9.5.13]{lazarsfeld}.

If $\textup{mult}_p\,\Delta=1$, then
$\textup{mult}_p\,\Delta_i=\deg(\Delta_i)$ for every $1 \le i\le k$. Thus
$\Delta_i$ is a cone in $\bP^r$ with vertex $p$ for every $1 \le i\le k$.

Let $P\subset \bP^r$ be a general plane passing through $p$.
Then $\textup{Supp}(\Delta)_{|P}$ is a union of $\sum_{1\le i\le k}\deg(\Delta_i)$ distinct lines through $p$. 
Let $S \to P$ be the blow-up of $P$ at $p$, with exceptional locus $E$. Then $S$ is a log resolution of the pair $(P,\Delta_{|P})$ and, $a(E,P,\Delta_{|P})=1-\deg(\Delta)\ge 0$. 
Thus, 
$(P,\Delta_{|P})$ is klt at $p$.
This implies that 
$(X,\Delta)$ is klt at $p$, by \cite[Corollary 9.5.11]{lazarsfeld}). 
\end{proof}

\subsection{$\bQ$-Fano varieties}

\begin{defn}
Let $X$ be a normal projective variety such that $K_X$ is $\bQ$-Cartier. 
We say that $X$ is \emph{$\bQ$-Fano} if $-K_X$ is ample.
 In this case, the \emph{index of $X$} is the largest positive rational number $i_X$ such that
$-K_{\sF} \sim_\bQ i_X H$ for an ample Cartier divisor $H$ on $X$.
\end{defn}

The following geometric property of klt $\bQ$-Fano varieties will be fundamental in our study of $\bQ$-Fano foliations.

\begin{thm}[{\cite[Corollaries 1.3 and 1.5]{hacon_mckernan}}]\label{thm:rc}
Let $(X,\Delta)$ be a klt pair.
If $-(K_X+\Delta)$ is nef and big, then $X$ is rationally connected.
\end{thm}

The following  result is certainly well known to experts. We include a proof for lack of adequate reference.

\begin{lemma}\label{lemma:Fano_lin_equ_and_num_equ}
Let $X$ be klt $\bQ$-Fano variety, and $\sL$ a line bundle on $X$. 
If $c_1(\sL)\equiv 0$, then $\sL\simeq\sO_X$.
In particular, if $\rho(X)=1$, then $\Pic(X) \simeq \bZ$.
\end{lemma}

\begin{proof}
Let $\pi : Y \to X$ be a resolution of singularities. Recall that
$X$ has rational singularities. The Leray spectral sequence implies that $h^i(X,\sO_X)=h^i(Y,\sO_Y)$ for all $i\ge 0$, and 
$h^i(X,\sL)=h^i(Y,\pi^*\sL)$ for all $i\ge 0$. By the Kawamata-Viehweg vanishing Theorem, 
$h^i(X,\sO_X)=0$ 
and $h^i(X,\sL)=0$ for all $i\ge 1$
(see \cite[Theorem 1.2.5]{kmm}).
 The Grothendieck-Riemann-Roch Theorem and the fact that $c_1(\sL)\equiv 0$ imply that 
$\chi(Y,\pi^*\sL)=\chi(Y,\sO_Y)$.
Thus 
$\chi(Y,\pi^*\sL)=1$, and $h^0(X,\sL)=h^0(Y,\pi^*\sL)=1$.
The same argument shows that $h^0(X,\sL^{\otimes -1})=1$ as well. The claim follows.
\end{proof}

We end this subsection by proving Theorem \ref{thm:ko}.

\begin{proof}[Proof of Theorem \ref{thm:ko}]
The result is clear if $n=1$
So we assume that $n\ge 2$.
Let $\pi : Y \to X$ be a resolution of singularities, with exceptional locus $E$. 
For every  $t\in\bZ$ we have
$$
\begin{array}{cccl}
h^n\big(Y,\pi^*\sL^{\otimes -t}\big) & \simeq  & 
h^0\big(Y,\sO_Y(K_Y)\otimes\pi^*\sL^{\otimes t}\big)
& \text{ by Serre duality}\\
& \le &   h^0\big(Y\setminus E,\sO_Y(K_Y)\otimes\pi^*\sL^{\otimes t}\big) 
& \\
& = & h^0\big(X\setminus \pi(E),\sO_X(K_X)\otimes\sL^{\otimes t}\big)
& \text{ since $Y\setminus E\simeq X\setminus \pi(E)$}\\
& = & h^0\big(X,\sO_X(K_X)\otimes\sL^{\otimes t}\big)
& \text{ since $\codim(\pi(E))\ge 2$}.
\end{array}
$$

Notice that $K_X+t c_1(\sL)\sim_\bQ -\Delta + (t-i) c_1(\sL)$.

Suppose that $i > n$. 
Then, for every $t\in\{1,\ldots, n\}$, we have 
$t-i < 0$, and thus
$h^n\big(Y,\pi^*\sL^{\otimes -t}\big)=h^0\big(X,\sO_X(K_X)\otimes\sL^{\otimes t}\big)=0$. 
Thus, by \cite[Theorem 2.2]{fujita89}, there is a birational morphism $\phi:X\to \bP^n$ such that $\sL\simeq \phi^*\sO_{\bP^n}(1)$. Since $\sL$ is ample, $\phi$ must be an isomorphism, proving (1). 

Suppose that $i = n$. 
Then, for every $t\in\{1,\ldots, n-1\}$, we have  $t-i < 0$, and thus
$h^n\big(Y,\pi^*\sL^{\otimes -t}\big)=h^0\big(X,\sO_X(K_X)\otimes\sL^{\otimes t}\big)=0$.
Moreover, 
$h^n\big(Y,\pi^*\sL^{\otimes -n}\big)=
h^0\big(X,\sO_X(K_X)\otimes\sL^{\otimes n}\big)$
and $K_X+nc_1(\sL)\sim_\bQ -\Delta$.
If 
$h^n\big(Y,\pi^*\sL^{\otimes -n}\big)=0$,
then 
$(X,\sL)\simeq 
(\bP^n,\sO_{\bP^n}(1))$ again 
by \cite[Theorem 2.2]{fujita89}.
Otherwise, $\Delta=0$,
$-K_X\sim_{\bZ} n c_1(\sL)$
and
$h^n\big(Y,\pi^*\sL^{\otimes -n}\big)=1$. 
By \cite[Theorem 2.3]{fujita89},
either $g(X,\sL)=1$, where $g(X,\sL)$ denotes the sectional genus of $(X,\sL)$, 
or there is a birational morphism $\psi:X\to Q_n$ onto a (possibly singular) quadric hypersurface in $\bP^{n+1}$ such that $\sL\simeq \psi^*\sO_{Q_n}(1)$.
In the latter case, $\psi$ must be an isomorphism 
since $\sL$ is ample. Finally note that
$g(X,\sL)=1+\frac{1}{2}(K_X+(n-1)c_1(\sL))\cdot c_1(\sL)^{n-1}
=1-\frac{1}{2}c_1(\sL)^{n}\neq 1$. This proves (2).
\end{proof}

\subsection{Singular del Pezzo varieties}

\begin{defn}[{\cite[Introduction]{fujita90}}] A \emph{del Pezzo variety} is a pair $(X,\sL)$, where $X$ is a (not necessarily normal) projective variety
and $\sL$ is an ample line bundle on $X$, such that:
\begin{enumerate}
\item $X$ is Gorenstein, 
\item $\omega_X\simeq \sL^{\otimes -(\dim(X)-1)}$,
\item $h^i(X,\sL^{\otimes t})=0$ for all $t \in \bZ$ and 
$0 < i < \dim(X)$.
\end{enumerate}

If $X$ is a Gorenstein klt $\bQ$-Fano variety, then, by the 
Kawamata-Viehweg vanishing Theorem
(see \cite[Theorem 1.2.5]{kmm}), condition (3) above is satisfied.
Therefore, $X$ is a (normal) del Pezzo variety if and only if condition (2) holds.
\end{defn}

Our classification of codimension $1$ del Pezzo foliations  in Theorem \ref{thm:classification_del_pezzo}(1)
will rely on the  following results.

\begin{prop}\label{prop:del_pezzo_degree} 
Let $(X,\sL)$ be a normal del Pezzo variety of dimension $n \ge 1$ and degree $d:=\sL^n$.
Suppose that $X$ is klt. Then $d \le 9$.
\end{prop}

\begin{proof} We may assume $d\ge 3$, in which case $\sL$ is very ample by \cite[Corollary 1.5]{fujita90}.

Suppose that $i_X\neq n-1$. Then
$(X,\sL)\simeq (\bP^3,\sO_{\bP^3}(2))$ by Theorem \ref{thm:ko}, and so $d=8$.   

If $X\subset \bP(H^0(X,\sL)^*)$ is not a cone, then $d\le 9$ by 
\cite{fujita1}, \cite{fujita2}, \cite{fujita3} and
\cite[2.9]{fujita86}
(see also \cite{fujita_classification}).

From now on we assume that 
$X\subset \bP(H^0(X,\sL)^*)$ is a (singular) cone over a normal projective variety $W$, with vertex a linear subspace $P$ of $\bP(H^0(X,\sL)^*)$.
We may assume that  $W$ is not a cone. 
Notice that $W$ is identified with a global complete intersection of general members of $|\sL|$. 
Let $\sM$ be the line bundle on $W$ induced by $\sL$.  
Then $W$ is Gorenstein and klt, 
$\sM^{\dim(W)}=d$, 
and $i_W=k(i_X-r)=k(\dim(W)-1)$, where $r=\dim(X)-\dim(W)$ and $k$ is a positive integer. 
If $\dim(W)=1$, then $W$ is a smooth curve of genus one, contradicting the fact that $X$ is klt. So $\dim(W)\ge 2$, and
$W$ is a normal del Pezzo variety of dimension $\ge 2$. 
Since $W$ is not a cone, we must have $d=\sM^{\dim(W)}\le 9$ as above, completing
the proof of the proposition.
\end{proof}

\begin{lemma}\label{lemma:vanishing_del_pezzo}
Let $(X,\sL)$ be a normal del Pezzo variety of dimension $n\ge 2$ and degree $\sL^n\le 2$. Then $h^0(X,\Omega_X^{[1]}\otimes\sL)=0$.
\end{lemma}

To prove  Lemma~\ref{lemma:vanishing_del_pezzo}, we will use the following description of  reflexive forms 
on double covers.

\begin{say}
Let $X$ be a smooth projective variety, and $D$ an effective reduced divisor on $X$ such that $\sO_X(D)\simeq \sR^{\otimes 2}$
for some line bundle $\sR$ on $X$.
Let $f:Y\to X$ be the double cover of $X$ ramified over $D$.
Then $Y$ is normal, and $f_*\Omega^{[1]}_Y$ is a reflexive sheaf by  \cite[Corollary 1.7]{hartshorne80}.
By \cite[Lemme 1.9]{viehweg}, 
$f_*\Omega^{[1]}_Y\simeq \Omega^{1}_{X}
 \oplus \Omega_{X}^{[1]}(\textup{log }D)\otimes\sR^{-1}$.
\end{say}

\begin{proof}[{Proof of Lemma~\ref{lemma:vanishing_del_pezzo}}]

Suppose that $\sL^n=2$. By \cite[Corollary 6.13]{fujita_classification},
the linear system $|\sL|$ induces a double cover $\pi : X \to \bP^n$ ramified over a quartic hypersurface $B\subset \bP^n$.
Then 
 $\pi_*\Omega^{[1]}_X\simeq \Omega^{1}_{\bP^n}
 \oplus \Omega_{\bP^n}^{[1]}(\textup{log }B)\otimes\sO_{\bP^n}(-2)$.
By the projection formula,
\begin{multline*}
h^0(X,\Omega_X^{[1]}\otimes\sL) = h^0(\bP^n,(\pi_*\Omega_X^{[1]})\otimes\sO_{\bP^n}(1)) 
= h^0(\bP^n,\Omega^{1}_{\bP^n}(1))+
h^0(\bP^n,\Omega_{\bP^n}^{[1]}(\textup{log }B)\otimes\sO_{\bP^n}(-1))=0.
\end{multline*}

\medskip

Suppose that $\sL^n=1$. By  \cite[Corollary 6.13]{fujita_classification}, $(X,\sL)$ admits the following description.
Set $Z:=\bP_{\bP^{n-1}}(\sO_{\bP^{n-1}}(2)\oplus\sO_{\bP^{n-1}})$, write $p:Z\to \bP^{n-1}$ for the natural morphism, and
$\sO_Z(1)$ for the tautological line bundle. 
Let $B\in |\sO_Z(3)|$ be a reduced divisor, $S\simeq \bP^{n-1}$ the section of $p$ corresponding to the
surjection $\sO_{\bP^{n-1}}(2)\oplus\sO_{\bP^{n-1}} \twoheadrightarrow \sO_{\bP^{n-1}}$, and set $D:=B+S$.
Observe that $S\cap B=\emptyset$, and $\sO_Z(D)\simeq \big( \sO_{Z}(2)\otimes p^*\sO_{\bP^{n-1}}(-1) \big)^{\otimes 2}$.
Let $f:Y\to Z$ be the double cover of $Z$ ramified over $D$, and set $E=f^{-1}(S)\simeq \bP^{n-1}$.
Then, for some choice of $B$ as above,  there is a birational morphism $\pi:Y\to X$ realizing $Y$ as the blowup of $X$ at a smooth point $x\in X$,
with exceptional divisor $E$. Moreover, the composite morphism $\rho : Y \to Z \to \bP^{n-1}$ is induced by the 
 linear system  $|\pi^*\sL\otimes\sO_Y(-E)|$.

We have
$$
\begin{array}{cccl}
h^0(X,\Omega^{[1]}_X\otimes\sL)  & =  & h^0(X,\Omega^{[1]}_X\otimes\sL\otimes\pi_*\sO_Y(E)) & \text{ since $E$ is exceptional}\\
						      & = & h^0(Y,\pi^*\Omega^{[1]}_X\otimes\pi^*\sL\otimes\sO_Y(E)) & \text{ by the projection formula}.
\end{array}
$$
The short exact sequence
$$0 \to \pi^*\Omega^{[1]}_X \to \Omega^{[1]}_Y \to \Omega^1_E\to 0$$ 
and the vanishing of $H^0(Y,\Omega^1_E\otimes\pi^*\sL\otimes\sO_Y(E))\simeq  
H^0(E,\Omega^1_E\otimes\sO_E(E))\simeq 
H^0(\bP^{n-1},\Omega^1_{\bP^{n-1}}\otimes\sO_{\bP^{n-1}}(-1))$
yield 
$$
h^0(Y,\pi^*\Omega^{[1]}_X\otimes\pi^*\sL\otimes\sO_Y(E))=h^0(Y,\Omega^{[1]}_Y\otimes\pi^*\sL\otimes\sO_Y(E)).
$$
Notice that $\pi^*\sL\otimes\sO_Y(E)\simeq f^*\sO_Z(1)\otimes \rho^*\sO_{\bP^{n-1}}(-1)$. Thus 
$$
\begin{array}{cccl}
h^0(Y,\Omega^{[1]}_Y\otimes\pi^*\sL\otimes\sO_Y(E)) & = & 
h^0(Y,\Omega^{[1]}_Y\otimes f^*\sO_Z(1)\otimes \rho^*\sO_{\bP^{n-1}}(-1))
& \ \\
& = & h^0(Z,f_*\Omega^{[1]}_Y\otimes \sO_Z(1)\otimes p^*\sO_{\bP^{n-1}}(-1))  
& \text{ by the projection formula}.\\
\end{array}
$$
Since 
$f_*\Omega^{[1]}_Y\simeq \Omega^{1}_{Z}
 \oplus \Omega_{Z}^{[1]}\big(\textup{log }(S+B)\big)\otimes\sO_{Z}(-2)\otimes p^*\sO_{\bP^{n-1}}(1)$, we conclude that 
 $$
 h^0(X,\Omega^{[1]}_X\otimes\sL)  \ = \  h^0(Z,  \Omega^{1}_{Z} \otimes \sO_Z(1)\otimes p^*\sO_{\bP^{n-1}}(-1)  ) + 
 h^0(Z, \Omega_{Z}^{[1]}\big(\textup{log }(S+B)\big)\otimes\sO_{Z}(-1)).
 $$
By restricting to a section of $p$ corresponding to the
surjection $\sO_{\bP^{n-1}}(2)\oplus\sO_{\bP^{n-1}} \twoheadrightarrow \sO_{\bP^{n-1}}(2)$,
we see that  $h^0(Z,\Omega^{1}_{Z}\otimes \sO_Z(1)\otimes p^*\sO_{\bP^{n-1}}(-1))=0$.
To see the vanishing of the second summand above, 
notice that $\Omega_{Z}^{[1]}\big(\textup{log }(S+B)\big)\otimes\sO_{Z}(-1)\subset 
\Omega_{Z}^{[1]}(\textup{log }B)\otimes \sO_Z(S) \otimes\sO_{Z}(-1)$.
On the other hand, $\sO_Z(S)\simeq \sO_Z(1)\otimes p^*\sO_{\bP^{n-1}}(-2)$, 
and so the desired vanishing follows from 
$$
h^0(Z,\Omega_{Z}^{[1]}(\textup{log }B)\otimes p^*\sO_{\bP^{n-1}}(-2))
=0.
$$
This completes the proof.
\end{proof}

\subsection{A vanishing result} \

Rationally chain connected varieties with mild singularities do not carry nonzero reflexive forms (see \cite{greb_kebekus_kovacs_peternell10}).
We show that, under suitable conditions, the same holds for twisted reflexive forms.

The following observation will be used throughout the paper: 
if $\sG$ is a reflexive sheaf of $\sO_X$-modules, and 
$i : U \hookrightarrow X$ is a dense open subset such that $\codim (X \setminus U )\ge 2$, then $\sG\simeq i_* \sG_{|U}$. 
In particular, any section of $\sG$ on $U$ extends to a section of $\sG$ on $X$.

\begin{lemma}\label{lemma:vanishing_weak_fano}
Let $(X,\Delta)$ be a
klt pair. Suppose 
that either $-(K_X+\Delta)$ is nef and big, or 
$K_X+\Delta\sim_{\bQ} 0$ and $\Delta$ is big.
Let $D$ be an integral divisor on $X$ such that $D\sim_{\bQ} 0$.
Then $h^0(X,\Omega_X^{q}[\otimes]\sO_X(-D))=0$
for every positive integer $q$.
\end{lemma}

\begin{proof}
Let $r$ be the smallest positive integer such that 
$r\,D \sim_\bZ 0$, and let $\tau:X' \to X$ be the corresponding cyclic cover (see \cite[Definition 2.52, Lemma 2.53]{kollar_mori}). 
Then $\tau$ is \'etale in codimension 1, and $\tau^{[*]}\sO_X(D)\simeq \sO_{X'}$.
Set $\Delta':=\tau^*\Delta$, and notice that $\tau^* K_X=K_{X'}$.
By \cite[Proposition 5.20]{kollar_mori}, $(X',\Delta')$ is klt.

Suppose that $-(K_X+\Delta)$ is nef and big.
Then so is  $-(K_{X'}+\Delta')\sim_{\bQ}- \tau^*(K_X+\Delta)$.
By Theorem~\ref{thm:rc}, $X'$ is rationally connected, and hence 
$h^0(X',\Omega_{X'}^{[q]})=0$ for any positive integer $q$ by \cite[Theorem 5.1]{greb_kebekus_kovacs_peternell10}.
On the other hand, 
any nonzero section of $\Omega_X^{q}[\otimes]\sO_X(-D)$ induces a nonzero section of $\Omega_{X'}^{[q]}$.
So we conclude that $h^0(X,\Omega_X^{q}[\otimes]\sO_X(-D))=0$.

Suppose now that $K_X+\Delta\sim_{\bQ} 0$ and $\Delta$ is big.
Then there exist ample $\bQ$-divisor $A$  and effective $\bQ$-divisor $N$ on $X$ such that 
$\Delta\sim_{\bQ}A+N$, and $(X,N)$ is klt.
Then $-(K_X+N)\sim_{\bQ} A$ is ample, and the claim follows from the previous case.
\end{proof}

We are ready to state and prove our vanishing result.

\begin{thm}\label{thm:main}
Let $(X,\Delta)$ be a klt pair with $\Delta$ big and $K_X+\Delta$ pseudo-effective.
Suppose that 
there exists an integral divisor $D$ on $X$ such that
$K_X+\Delta\sim_{\bQ}D$. Then 
$h^0(X,\Omega_X^{q}[\otimes]\sO_X(-D))=0$ 
for every positive integer $q$.
\end{thm}

\begin{proof}

We denote the dimension of $X$ by $n$.

Consider the integral Weil divisor $\Gamma=D-K_X\sim_{\bQ}\Delta$ on $X$. 

Since
$K_X+\Delta$ is pseudo-effective, the pair $(X,\Delta)$ has a minimal model
$\phi : X \map Y$ (\cite[Corollary 1.4.2]{BCHM}). 
Set $\Delta_Y:=\phi_*\Delta$, $D_Y:=\phi_*D$ and 
$\Gamma_Y:=\phi_*\Gamma$, and notice that $K_Y=\phi_*K_X$.
Then $(Y, \Delta_Y)$ is klt, $\Delta_Y$ is big, and 
$K_Y+\Delta_Y\sim_\bQ D_Y$ is nef.
By the base point free Theorem (\cite[Corollary 3.9.2]{BCHM}),  
$K_Y+\Delta_Y$ is semiample, and thus $\kappa:=\kappa(\sO_{Y}(K_Y+\Delta_Y)) \ge 0$.

Suppose first that $\kappa =0 $. 
Since $K_Y+\Delta_Y$ is semiample,  $K_Y+\Delta_Y\sim_\bQ 0$.
By Lemma \ref{lemma:vanishing_weak_fano},  for every positive integer $q$,
$h^0(Y,\Omega_Y^{q}[\otimes]\sO_Y(-D_Y))=0$.
Since $\phi^{-1}$ does not contract any divisor, we conclude that
$h^0(X,\Omega_{X}^{q}[\otimes]\sO_{X}(-D))=0$
for every positive integer $q$.

Suppose from now on that $\kappa \ge 1$, and
let $\pi : Y \to W$ be the surjective morphism onto a normal variety defined by the linear system
$|mD_Y|$ for $m$  sufficiently large and divisible. 
Denote by $W_0$ the smooth locus of $W$, and note that $\dim(W)=\kappa$.

Suppose that 
$h^0(X,\Omega_X^{q}[\otimes]\sO_X(-D))\neq0$
for some $q \ge 1$.
Let $\xi:\sO_X(D) \hookrightarrow \Omega^{[q]}_X$ be the nonzero map
associated to a nonzero global section of $\Omega_X^{q}[\otimes]\sO_X(-D)$.
Since $\phi^{-1}$ does not contract any divisors, $\xi$ induces a nonzero map
$\xi_Y:\sO_{Y}(D_Y) \hookrightarrow \Omega^{[q]}_{Y}$.
By restricting $\xi_Y$ to the smooth locus $Y_0$ of $Y$,
considering the perfect pairing 
$\tau_0 : \Omega^{n-q}_{Y_0}\otimes\Omega^q_{Y_0} \to \Omega^n_{Y_0}$,
and extending back to $Y$, we see that 
$\xi_Y$ induces a nonzero map 
$\eta_Y : \Omega^{[n-q]}_{Y} \to \sO_{Y}(-{\Gamma_Y})$.

\begin{claim}
Over $\pi^{-1}(W_0)$, the map ${\eta_Y}$ factors through the canonical map 
$\Omega^{[n-q]}_{Y}\to \Omega^{[n-q]}_{Y/W}$. 
\end{claim}
\noindent To prove the claim, it is enough to show this factorization over $\pi^{-1}(W_0)\cap Y_0$.
Fix $y\in Y_0$ such that $w=\pi(y)\in W_0$.
At $y$, the kernel of 
$\Omega^{n-q}_{Y_0}\to \Omega^{n-q}_{Y_0/W}$ is generated
by $(n-q)$-forms of the form  
$\beta\wedge\pi^*\gamma$, where $\beta$ 
is a local section of $\Omega^{n-q-1}_{Y_0}$  at $y$, and  $\gamma$
is a local section of  $\Omega^1_{W}$ at $w$.
So we must show that ${(\eta_Y)_y}(\beta\wedge\pi^*\gamma)=0$. 
Let $\omega$ 
be a local generator of $\Omega^n_{Y_0}$ at $y$,
and $e$  a local  generator of $\sO_{Y}(\Gamma_Y)$ at $y$.
Then $\omega\otimes e$ is a local generator of $\sO_{Y}(D_Y)$ at $y$, and 
$$
{(\eta_Y)_y}(\beta\wedge\pi^*\gamma)\otimes\omega\otimes e 
= {(\tau_0)_y}\big({(\xi_Y)}_y(\omega\otimes e)\wedge\beta\wedge\pi^*\gamma\big).
$$
In order to prove the claim, we will show that 
\begin{equation}\label{eq1}
{(\xi_Y)}_y(\omega\otimes e)\wedge\pi^*\gamma=0.
\end{equation}
Given $s_0\in H^0(Y,\sO_Y(mD_Y))$ generating $\sO_Y(mD_Y)$ at $y$, 
we can find $(h_1,\ldots,h_\kappa)$ a local system of coordinates at $w$
such that $s_i=\pi^*h_i\,s_0\in H^0(Y,\sO_Y(mD_Y))$. 
In order to prove \eqref{eq1}, it is enough to show that, 
for every $1 \le i \le \kappa$, 
${(\xi_Y)}_y(\omega\otimes e)\wedge  \pi^*dh_i=0$.
This follows from the proof of the Bogomolov-Sommese vanishing Theorem (see \cite[Theorem III]{viehweg82}).
For the convenience of the reader, we sketch the proof. 
Let $p:T\to Y$ be a resolution of singularities of $Y$.
By \cite[Theorem 4.3]{greb_kebekus_kovacs_peternell10}, there exits
a nonzero map
$p^*\Omega^{[q]}_Y\to \Omega^{q}_{T}$. Thus, there exists a nonzero map $p^{[*]}\sO_Y(D_Y)\to \Omega^{[q]}_{T}\simeq \Omega^{q}_{T}$.
Replacing $Y$ with $T$, 
and $\sO_Y(D_Y)$ with $p^{[*]}\sO_Y(D_Y)$, 
we may assume that $Y$ is smooth.
By taking successively $m$-th root out of the $s_i$'s we may also assume that $m=1$,
and that $s_0= \omega\otimes e$. 
For $0 \le i \le \kappa$, set 
$\alpha_i:=\xi_Y(s_i)\in H^0(Y,\Omega_Y^1)$. 
Then $d\alpha_i=0$. On the other hand, 
$d\alpha_i=-\pi^*dh_i \wedge\alpha_0$.
This proves the claim.

Next we show that $q=\kappa$. 
Let $w\in W_0$ be a general point, and set $F:=\pi^{-1}(w)$. 
Then 
$(F,{\Delta_Y}_{|F})$ is klt,  and the adjunction formula gives   
${D_Y}_{|F}\sim_\bQ K_F+{\Delta_Y}_{|F}\sim_\bQ 0$.
Since $\eta_Y$ factors through the canonical map 
$\Omega^{[n-q]}_{Y}\to \Omega^{[n-q]}_{Y/W}$,  it induces a nonzero map
$\Omega^{[n-q]}_{F} \to \sO_F(-{\Gamma_Y}_{|F})$.
Thus $h^0(F,\Omega^{\dim(F)-n+q}_F[\otimes] \sO_F(-{D_Y}_{|F}))\neq 0$, and
Lemma \ref{lemma:vanishing_weak_fano} yields
$\dim(F)=n-q$. Hence $q=\kappa$. 

Let $C \subset W$ be a general complete intersection curve. 
Then $C \subset W_0$, $C$ is smooth, and $Y_C=\pi^{-1}(C)\subset Y$ is a normal projective variety equipped with a flat morphism 
$\pi_C:Y_C\to C$ whose general fibers are reduced.
Moreover, $\eta_Y$ restricts to a nonzero map
$\eta_{Y_C} : \Omega^{[n-q]}_{Y_C/C} \to 
\sO_{Y_C}(-{\Gamma_Y}_{|Y_C})$.
By \cite[Theorem 2.1']{bosch95}, there exists
a  smooth complete curve $B$, together with
a finite morphism $B \to C$, such that $\varpi:Z \to B$ has reduced fibers, where
$Z$ is the normalization of the fiber product $Y_C\times_C B$, and $\varpi:Z \to B$ is the natural morphism. 
Let $r : Y_C\times_C B \to Y_C$ denote the natural morphism, and set $(Y_C\times_C B)_0:=r^{-1}(Y_C\setminus \textup{Sing}(X))$. The map $\eta_{Y_C}$ induces 
a nonzero map $\Omega^{[n-q]}_{(Y_C\times_C B)_0/B} \to 
\sO_{(Y_C\times_C B)_0}(-{\Gamma_Y}_{|(Y_C\times_C B)_0})$.
By \cite[Proposition 4.5]{adk08}, this extends to a 
nonzero map
$$\eta_Z : \Omega^{[n-q]}_{Z/B} \to \sO_{Z}(-{\Gamma_Y}_{|Z}),$$
using the fact that 
$(Y_C\times_C B)\setminus (Y_C\times_C B)_0$ has codimension at least two in $Y_C\times_C B$.
By \cite[Lemme 4.4]{druel99}, 
$\Omega^{[n-q]}_{Z/B}\simeq \sO_{Z}(K_{Z/B})$.
If we denote by $F$ a general fiber of $\varpi$, then
$K_{F}+{\Delta_Y}_{|F}\sim_\bQ 0$ and
${\Gamma_Y}_{|Z}\sim_\bQ {\Delta_Y}_{|Z}$.
Thus there exists an effective divisor $\Psi$ on $Z$ such that $\varpi(\Supp(\Psi))\subsetneq B$ and
$$-(K_{Z/B}+\Psi)={\Gamma_Y}_{|Z}\sim_\bQ {\Delta_Y}_{|Z}.$$
Since ${\Delta_Y}_{|Z}$ is big, we can write 
${\Delta_Y}_{|Z}\sim_\bQ A+N$, with
$A$ an ample $\bQ$-divisor  and $N$ an effective $\bQ$-divisor.
Since $(Z,{\Delta_Y}_{|Z})$ is klt over the generic point of $B$,
so is $(Z,(1-\varepsilon){\Delta_Y}_{|Z}+\varepsilon N)$  for $\varepsilon >0$ sufficiently small. 
On the other hand, 
$$-(K_{Z/B}+\Psi+(1-\varepsilon){\Delta_Y}_{|Z}+\varepsilon N)\sim_\bQ \varepsilon A$$
is ample, contradicting \cite[Theorem 3.1]{adk08}. 
We conclude that $h^0(X,\Omega_X^{q}[\otimes]\sO_X(-D))=0$.
\end{proof}


\section{Foliations} \label{section:foliations}

\subsection{Foliations and Pfaff fields}

\begin{defn}
Let $X$ be a normal variety.
A \emph{foliation} on $X$ is a nonzero coherent subsheaf $\sF\subsetneq T_X$ satisfying
\begin{enumerate}
	\item $\sF$ is closed under the Lie bracket, and
	\item $\sF$ is saturated in $T_X$ (i.e., $T_X / \sF$ is torsion free).
\end{enumerate}

The \textit{rank} $r$ of $\sF$ is the generic rank of $\sF$.

The \textit{canonical class} $K_{\sF}$ of $\sF$ is any Weil divisor on $X$ such that 
$\sO_X(-K_{\sF})\simeq \det(\sF)$. 
\end{defn}

\begin{say}[Foliations defined by $q$-forms] \label{q-forms}
Let $\sF$ be a codimension $q$ foliation on an $n$-dimenional normal variety $X$.
The \emph{normal sheaf} of $\sF$ is $N_\sF:=(T_X/\sF)^{**}$.
The $q$-th wedge product of the inclusion
$N^*_\sF\into (\Omega^1_X)^{**}$ gives rise to a nonzero global section 
 $\omega\in H^0\big(X,\Omega^{q}_X[\otimes] \det(N_\sF)\big)$
 whose zero locus has codimension at least $2$ in $X$. 
Such $\omega$ is \emph{locally decomposable} and \emph{integrable}.
To say that $\omega$ is locally decomposable means that, 
in a neighborhood of a general point of $X$, $\omega$ decomposes as the wedge product of $q$ local $1$-forms 
$\omega=\omega_1\wedge\cdots\wedge\omega_q$.
To say that it is integrable means that for this local decomposition one has 
$d\omega_i\wedge \omega=0$ for all  $i\in\{1,\ldots,q\}$. 

Conversely, let $\sL$ be a reflexive sheaf of rank $1$ on $X$, and 
$\omega\in H^0(X,\Omega^{q}_X[\otimes] \sL)$ a global section
whose zero locus has codimension at least $2$ in $X$.
Suppose that $\omega$  is locally decomposable and integrable.
Then  we define 
a foliation of rank $r=n-q$ on $X$ as the kernel
of the morphism $T_X \to \Omega^{q-1}_X[\otimes] \sL$ given by the contraction with $\omega$. 
These constructions are inverse of each other. 
 \end{say}

\begin{defn}\label{def:gorenstein}
A foliation $\sF$ on a normal variety is said to be \emph{$\bQ$-Gorenstein} if 
its canonical class $K_{\sF}$ is $\bQ$-Cartier.
It is said to be Gorenstein if $K_{\sF}$ is Cartier.
\end{defn}

\begin{defn}\label{def:pfaff}
Let $X$ be a variety, and $r$ a  positive integer.
A  \emph{Pfaff field of rank r} on $X$ is a nonzero map
$\eta : \Omega^r_X\to \sL$, where  $\sL$ is a reflexive sheaf of rank 1
on $X$ such that $\sL^{[m]}$ is invertible for some integer $m\ge 1$. 
The \textit{singular locus} $S$ of $\eta$
is the closed subscheme of $X$ whose ideal sheaf $\sI_S$ is the image of
the induced map $\Omega^r_X[\otimes] \sL^*\to \sO_X$.

A closed subscheme $Y$ of $X$ is said to be \textit{invariant} under $\eta$ if 
\begin{enumerate}
	\item no irreducible component of $Y$ is contained in the singular locus of $\eta$, and
	\item the restriction 
	$\otimes^m\eta_{|Y} : {\otimes^m \Omega^r_X}_{|Y}\to{\sL^{[m]}}_{|Y}$ factors through the natural map
	${\otimes^m\Omega^r_X}_{|Y}\to\otimes^m\Omega^r_Y$, where $m\ge 1$ is such that $\sL^{[m]}$ is invertible.
	In other words, there is a commutative diagram 

\centerline{
\xymatrix{
\otimes^m{\Omega^r_X}_{|Y} \ar[r]^{\otimes^m\eta_{|Y}}\ar[d] & {\sL^{[m]}}_{|Y} \ ,\\
\otimes^m\Omega^r_Y\ar[ru]
}
}
\noindent where the vertical map is the natural one. 
\end{enumerate}
\end{defn}

Suppose that $Y$ is reduced and set $Y_0:=Y\setminus\textup{Sing}(Y)$. Observe that
$\sL_{|Y_0}$ is locally free (see \cite[Proposition 1.9]{hartshorne80}).
Suppose that condition (1) above is satisfied.
Then, $Y$ is invariant under $\eta$ if and only if 
the restriction 
$\eta_{|Y_0} : {\Omega^r_X}_{|Y_0}\to\sL_{|Y_0}$ factors through the natural map
${\Omega^r_X}_{|Y_0}\to{\Omega^r_Y}_{|Y_0}$.

Notice that a $\bQ$-Gorenstein foliation $\sF$ of rank $r$  on normal variety $X$ naturaly gives rise to a Pfaff field of rank $r$ on $X$:
$$
\eta:\Omega_X^r=\wedge^r(\Omega_X^1) \to \wedge^r(T_X^*) \to \wedge^r(\sF^*) \to \det(\sF^*)\simeq\det(\sF)^{*}=\sO_X(K_\sF).
$$

\begin{defn}\label{defn:sing_locus} \label{def:regular}
Let  $\sF$ be a $\bQ$-Gorenstein foliation on a normal variety $X$.
The \textit{singular locus} of  $\sF$ is defined to be
the singular locus $S$ of the associated Pfaff field. 
We say that $\sF$ is \textit{regular at a point} $x\in X$ if $x\not\in S$. 
We say that $\sF$ is \textit{regular} if $S=\emptyset$.
\end{defn}

\begin{exmp}[{\cite[Remark 3.12]{fano_foliation}}]
The notion of regularity of foliations discussed above do not say anything about the singularities 
of the ambient space.
For instance, let $Y$ be a smooth variety, $T$ any normal variety, and set $X:=Y\times T$, 
with natural projection $p:X\to Y$.
Set $\sF:= p^*T_Y\subset T_X$.
Then $\sF$ is a regular Gorenstein foliation, while $X$ may be very singular.
\end{exmp}

Our definition of Pfaff field is more general than the one usually found in the literature, where $\sL$ is required to be
invertible. This generalization is needed in order to treat  $\bQ$-Gorenstein foliations whose canonical classes are not Cartier.
Under the more restrictive definition, it was proved in \cite[Proposition 4.5]{adk08} that Pfaff fields extend uniquely to the normalization.
The next lemma generalizes this to our current setting.

\begin{lemma}\label{lemma:extensionpfafffields}
Let $X$ be a normal variety, and
$\sL$ a torsion free sheaf of rank $1$ on $X$ such 
that $\sL^{[m]}$ is locally free for some integer $m \ge 1$.
Let  $r$ be a positive integer, and
$\eta : \Omega^r_X\to \sL$ a Pfaff field.
Let $F \subset X$ be an invariant integral closed subscheme, and $\tilde e : \tilde F\to X$ its normalization.
Then the map $\otimes^m\Omega^r_F \to {\sL^{[m]}}_{|F}$
extends uniquely to a generically surjective map
$\otimes^m\Omega^r_{\tilde F}\to \tilde e^*\sL^{[m]}$.
\end{lemma}

\begin{proof}
It is enough to prove the claim on some open cover of $F$.
Thus, by replacing $X$ with an open subset if necessary, we 
may assume that $\sL^{[m]}\simeq\sO_X$.
Let $\pi:Y \to X$ be the corresponding 
cyclic cover (see \cite[Definition 2.52, Lemma 2.53]{kollar_mori}), 
with Galois group
$G$. Then $\pi$ is \'etale in codimension one, and
$\pi^{[*]}\sL\simeq \sO_Y$ is locally free.
Thus 
$\eta$ induces a Pfaff field $\zeta :
\Omega^{r}_Y\to \pi^{[*]}\sL$
on $Y$, and $Z:=\pi^{-1}(F)\subset Y$ is invariant under $\zeta$.
Let $\tilde f :\tilde Z\to Y$ be the normalization of $Z$, and 
denote by $\tilde \pi : \tilde Z \to \tilde F$ the induced morphism. 
Then
$\zeta$ induces a map
$\Omega^r_Z\to {\pi^{[*]}\sL}_{|Z}$, which extends uniquely to a generically surjective map
$\Omega^r_{\tilde Z}\to \tilde f^*\pi^{[*]}\sL$ by
\cite[Proposition 4.5]{adk08}. 
So we obtain a nonzero map
$$\xi:\otimes ^m\tilde \pi ^*\Omega^r_{\tilde F}
\overset{\otimes^m d\tilde\pi}{\longrightarrow}
\otimes^m\Omega^r_{\tilde Z}
\to \otimes^m\tilde f^*\pi^{[*]}\sL
\simeq \tilde \pi^* \tilde e^* \sL^{[m]}.$$
Observe that the natural action of $G$ on $Z$ extends to an action 
on $\tilde Z$ such that $\tilde Z/G\simeq \tilde Y$. Moreover,  
$\otimes^m\tilde \pi^*\Omega^r_{\tilde F}$ and
$\tilde \pi^* \tilde e^* \sL$ are naturally $G$-linearized sheaves
on $\tilde Z$
and $\xi$ is $G$-equivariant. Our claim follows. 
\end{proof}

\subsection{Algebraically integrable foliations}\label{section:algebraic_foliations}

\begin{defn}
Let $X$ be a normal variety. A foliation $\sF$ on $X$ is said to be
\emph{algebraically integrable} if  the leaf of $\sF$ through a general point of $X$ is an algebraic variety. 
In this situation, by abuse of 
notation we often use the word \textit{leaf} to mean the closure in $X$ of a leaf of $\sF$. 
\end{defn}

The following Lemma is well known to experts. See
for instance \cite[Lemma 3.2]{fano_foliation}.

\begin{lemma}\label{lemma:leaffoliation}
Let $X$ be a normal projective variety, and $\sF$ an algebraically integrable foliation on $X$.
There is a unique irreducible closed subvariety $W$ of $\Chow(X)$ 
whose general point parametrizes the closure of a general leaf of $\sF$
(viewed as a reduced and irreducible cycle in $X$). In other words, if 
$U \subset W\times X$ is the universal cycle, with universal morphisms
$\pi:U\to W$ and $e:U\to X$,
then $e$ is birational, and, for a general point $w\in W$, 
$e\big(\pi^{-1}(w)\big) \subset X$ is the closure of a leaf of $\sF$.
\end{lemma}

\begin{notation}
We say that the subvariety $W$  provided by Lemma~\ref{lemma:leaffoliation} is 
\emph{the closure in $\Chow(X)$ of the subvariety parametrizing general leaves of $\sF$}.
We call the induced rational map $X\dashrightarrow W$ 
a \emph{rational first integral for $\sF$}.
\end{notation}

In \cite[Definition 3.4]{fano_foliation} we introduced the notion of general log leaf for algebraically integrable Gorenstein foliations.
Next we extend this to the $\bQ$-Gorenstein case.

\begin{defn}\label{defn:log_leaf}
Let $X$ be a normal projective variety, $\sF$ a $\bQ$-Gorenstein algebraically integrable foliation of rank $r$ on $X$, and 
$\eta:\Omega_X^{r} \to \sO_X(K_\sF)$ the corresponding Pfaff field.
Let $F\subset X$ be the closure of a general leaf of $\sF$, and $\tilde e:\tilde F\to X$
the normalization of $F$. Let $m \ge 1$ be the Cartier index of $K_\sF$, \textit{i.e.}, the smallest positive integer $m$ such that $mK_\sF$ is Cartier. 
By \cite[Lemma 2.7]{fano_foliation}, $F$ is invariant under $\eta$, \textit{i.e.}, the restriction 
$\otimes^m\eta_{|F} : {\otimes^m \Omega^r_X}_{|F}\to\sO_X(mK_\sF)_{|F}$ factors through the natural map
	${\otimes^m\Omega^r_X}_{|F}\to\otimes^m\Omega^r_Y$.
By Lemma \ref{lemma:extensionpfafffields}, the induced map
$\otimes^m\Omega^r_F \to \sO_X(mK_\sF)_{|F}$
extends uniquely to a generically surjective map
$\otimes^m\Omega^r_{\tilde F}\to \tilde e^*\sO_X(mK_\sF)$.
Hence there is a canonically defined effective Weil $\bQ$-divisor $\tilde \Delta$ on $\tilde F$ such that 
$mK_{\tilde F}+m\tilde \Delta\sim_\bZ \tilde e^*mK_\sF$. 
Namely, $m\tilde \Delta$ is the divisor of zeroes of $\tilde\eta$.

We call the pair $(\tilde F,\tilde \Delta)$ a \emph{general log leaf} of $\sF$.
\end{defn}

\begin{rem}\label{remark:definition_log_leaf}
Let $X$ be a normal projective variety,  and $\sF$ an algebraically integrable  $\bQ$-Gorenstein foliation of rank $r$ on $X$.
Let $W$ be the closure in $\Chow(X)$ of the subvariety parametrizing general leaves of $\sF$, and
$U \subset W\times X$ the universal cycle. Denote by $e:U\to X$ the natural morphism. 
Let $m$ be the Cartier index of $K_\sF$.
Then $\sF$ induces a map 
$\otimes^m\Omega_U^{r} \to e^*\sO_X(mK_\sF)$, which factors through 
the natural morphism $\otimes^m\Omega_U^{r} 
\twoheadrightarrow \otimes^m\Omega_{U/W}^{r}$ (see \cite[Remark 3.8 and Lemma 3.2]{fano_foliation}).

Let $\tilde W$ and $\tilde U$ be the normalizations of $W$ and $U$, respectively, and denote by 
$\tilde \pi:\tilde U\to \tilde W$ and $\tilde e:\tilde U\to X$ the induced morphisms. 
By Lemma \ref{lemma:extensionpfafffields}, the map 
$\otimes^m\Omega_U^{r} \to e^*\sO_X(mK_\sF)$
extends uniquely to a nonzero map
$\otimes^m\Omega_{\tilde U}^{r} \to {\tilde e}^*\sO_X(mK_\sF)$.
As before, this map factors through the natural map
$\otimes^m\Omega_{\tilde U}^{r} \twoheadrightarrow 
\otimes^m\Omega_{\tilde U/\tilde W}^{r}$,
yielding a generically surjective map 
$$
\otimes^m\Omega_{\tilde U/\tilde W}^{r}\to {\tilde e}^*\sO_X(mK_\sF).
$$
Thus there is a canonically defined effective Weil $\bQ$-divisor $\Delta$ on $\tilde U$ such that 
$K_{\tilde U/\tilde W}+ \Delta \sim_\bQ  {\tilde e}^*K_\sF$.

Let $w$ be a general point of $\tilde W$, set $\tilde U_w := \tilde \pi^{-1}(w)$ and $\Delta_w:=\Delta|_{\tilde U_w}$.
Then $(\tilde U_w, \Delta_w)$ coincides with the general log leaf $(\tilde F,\tilde \Delta)$ defined above.

The same construction can be carried out for any base change $V \to W$.
\end{rem}

Next we define notions of singularity for $\bQ$-Gorenstein algebraically integrable foliations
according to the singularity type of their general log leaf.

\begin{defn}\label{def:sing_fol}
Let $X$ be a normal projective variety, $\sF$ a $\bQ$-Gorenstein algebraically integrable foliation  on $X$,
and $(\tilde F,\tilde \Delta)$ its general log leaf. 
We say that $\sF$ has \emph{log terminal singularities along a general leaf} if
$(\tilde F,\tilde \Delta)$ is log terminal.
\end{defn}

\begin{prop}
\label{prop:foliation_algebraically_integrable_klt_not_weak_fano}
Let $\sF$ be an algebraically integrable $\bQ$-Gorenstein  foliation on a  normal projective variety $X$. Suppose
that $\sF$ has log terminal singularities along a general leaf.
Then $\det(\sF)$ is not nef and big. 
\end{prop}

\begin{proof}
The proof is the same as the proof of  \cite[Proposition 5.8]{fano_foliation}.
One only needs to replace the use of  
\cite[Proposition 4.5]{adk08} with Lemma~ \ref{lemma:extensionpfafffields}.
\end{proof}

\subsection{$\bQ$-Fano foliations}

\begin{defn}\label{def:fano_foliation}\label{def:index_fano_foliation}\label{def:fano_pfaff}\label{def:index_fano_pfaff}
Let $X$ be a normal projective variety, and 
$\sF$ a $\bQ$-Gorenstein foliation  of rank $r$ on $X$.

We say that $\sF$ is a \emph{$\bQ$-Fano foliation} if $-K_{\sF}$ is ample.
In this case, the index of $\sF$ is the largest positive rational number $i_{\sF}$ such that
$-K_{\sF} \sim_\bQ i_{\sF} H$ for a Cartier divisor $H$ on $X$.

We say that a $\bQ$-Fano foliation  $\sF$ is a \textit{del Pezzo foliation} if $r\ge 2$ and $i_{\sF} = r-1$.
\end{defn}

\begin{rem}
Let $\sF$ be $\bQ$-Fano foliation of rank $r$ and index $i_{\sF}$ on a normal projective variety $X$, 
and let $H$ be an ample divisor on $X$ such that $-K_{\sF} \sim_\bQ i_{\sF}H$. 
Suppose that $K_\sF$ is Cartier. 
In general, we do not have $-K_{\sF} \sim_\bZ i_{\sF} H$.
However, this  holds if $X$ is a klt $\bQ$-Fano variety by Lemma \ref{lemma:Fano_lin_equ_and_num_equ}. 
\end{rem}

\begin{proof}[{Proof of Theorem \ref{thm:main_foliation}}]
Let $X$ be a klt projective variety, and $\sF\subsetneq T_X$ a $\bQ$-Fano foliation.
We assume to the contrary that $K_X-K_\sF$ is pseudo-effective.
Set $n:=\dim(X)$, and denote by $r$ the rank of $\sF$,  $1 \le r \le n-1$. 
The $r$-th wedge product of the inclusion 
$\sF\subsetneq T_X$ gives rise to 
a nonzero map 
$\xi : \sO_X(K_X)\otimes\det(\sF) \hookrightarrow \Omega^{[n-r]}_X$. Thus
$h^0(X,\Omega^{n-r}_X[\otimes]\sO_X(-K_X+K_\sF)\neq 0$. 
Choose a $\bQ$-divisor $\Delta\sim_{\bQ} -K_\sF$ such that the pair $(X,\Delta)$ is klt, and set
$D:=K_X-K_\sF$. 
Then $D$ is an integral divisor on $X$, $K_X+\Delta\sim_{\bQ}D$, and 
$h^0(X,\Omega^{n-r}_X[\otimes]\sO_X(-D)\neq 0$.
This contradicts Theorem \ref{thm:main} since $n-r\ge 1$.
\end{proof}

The following is an immediate consequence of Theorem \ref{thm:main_foliation}.

\begin{cor}\label{cor:index} \label{cor:fano_foliation_picard_number_one}
Let $X$ be a klt projective variety with Picard number $\rho(X)=1$, and $\sF$ a $\bQ$-Fano foliation of index $i_\sF$ on $X$. 
Then $X$ is a $\bQ$-Fano variety, and its index satisfies $i_X > i_\sF$.
\end{cor}


\section{The Kobayashi-Ochiai Theorem for foliations}\label{section:ko}

For Fano foliations on smooth projective varieties, the index is at most the rank, and equality holds only for degree $0$
foliations on $\bP^n$ (\cite[Theorem 2.10]{fano_foliation}). 
The goal of this section is to extend this result to the singular case.
We start by noticing that by allowing singularities we get more examples of $\bQ$-Fano foliations with index equal to the rank,
which we describe now.

\begin{defn}\label{Normal generalized cones}
Let $\sM$ be an ample line bundle on a normal projective variety
$Z$, and $r$ a positive integer. 
Consider the $\bP^r$-bundle  $Y=\bP_Z(\sM\oplus \sO_Z^{\oplus r})$, with natural projection
$\pi : Y\to Z$. The tautological line bundle $\sO_Y(1)$ is semiample on $Y$. 
For $m\gg 0$, the linear system $|\sO_Y(m)|$ induces a birational morphism $e : Y \to X$ onto a normal projective variety. 
The morphism $e$ contracts the divisor $E=\bP_Z(\sO_Z^{\oplus r})\subset Y$ corresponding to the projection 
$\sM\oplus \sO_Z^{\oplus r}\twoheadrightarrow \sO_Z^{\oplus r}$ onto  $e(E)\simeq\bP^{r-1}$, and 
induces an isomorphism $Y\setminus E \simeq X\setminus e(E)$.
Following \cite{beltrametti_sommese}, we call $X$ the \emph{normal generalized cone over the base $(Z,\sM)$
with vertex $e(E)\simeq \bP^{r-1}$}.
If $r=1$, then $X$ is called a normal cone over the base  $(Z,\sM)$.
\end{defn}

\begin{rem} \label{rem:Normal generalized cones}
Let the notation be as in Definition~\ref{Normal generalized cones}.

If $\rho(Z)=1$ and $Z$ is $\bQ$-factorial, then the same holds for $X$.

If $X$ is factorial, then so is $Z$.

Suppose that  $Z$ is  klt, and $K_Z\sim_\bQ -(1+a) c_1(\sM)$ for some $a \in \bQ$.
Then $Y$ is also klt, and
$K_Y\sim_\bQ\pi^*K_Z-(r+1)c_1(\sO(1))+\pi^*c_1(\sM) \sim_\bQ -(r+1)c_1(\sO_Y(1))-a\pi^*c_1(\sM)$. 
Let  $\sL$ be a line bundle on $X$ such that $e^*\sL\simeq \sO_Y(1)$. Then
$e_*\pi^*c_1(\sM)\sim_\bZ c_1(\sL)$, and 
$E\sim_\bZ c_1(\sO_Y(1))-\pi^*c_1(\sM)\sim_\bZ e^*c_1(\sL)-\pi^*c_1(\sM)$. Thus 
$K_X=e_*K_Y\sim_\bQ -(r+a+1) c_1(\sL)$ is $\bQ$-Cartier, and
$K_Y\sim_\bQ e^*K_X+aE$. 
By \cite[Lemma 3.10]{kollar97}), we conclude that 
$X$ is klt if and only if $a>-1$.
\end{rem}

\begin{say}\label{say:Normal generalized cones}
Let the notation be as in Definition~\ref{Normal generalized cones} and Remark~\ref{rem:Normal generalized cones}.
Let $\sF$ be the foliation of rank $r$ on $X$ induced by the rational map $X\dashrightarrow Z$.
Then $-K_{\sF}=r\sL$.
\end{say}

We will show that, under suitable conditions, the $\bQ$-Fano foliations described in \ref{say:Normal generalized cones} are 
the only ones for which the index equals the rank 
(Theorem~\ref{thm:codimension_r_Fano_index__equal_rank_picard_number_1}).

\begin{prop}\label{prop:ko_picard_number_at_least_2}
Let $X$ be a klt projective variety of dimension $n$.
If $X$ admits a $\bQ$-Fano foliation $\sF\subsetneq T_X$
of index $i_{\sF}\ge n-1$,
then $\rho(X)=1$.
\end{prop}

\begin{proof}
Let $\sL$ be an ample line bundle on $X$ such that $-K_\sF\sim_\bQ i_\sF c_1(\sL)$. 
By Theorem \ref{thm:main_foliation}, $K_X+i_\sF c_{1}(\sL)$ is not nef.
If $\rho(X)\ge 2$, then, by  \cite[Theorem 2.1]{andreatta}, 
there is a surjective morphism $p:X\to Y$ onto a smooth projective curve $Y$ such that 
$(F,\sL_{|F})\simeq (\bP^{n-1},\sO_{\bP^{n-1}}(1))$ for  a general fiber  $F$ of  $p$. 
Let $\ell  \subset F \simeq \bP^{n-1}$ be a general line. Then $X$ is smooth in a neighborhood of $\ell$ and
$\sF_{|\ell}\subset {T_X}_{|\ell}\simeq \sO_{\bP^1}(2)\oplus\sO_{\bP^1}(1)^{\oplus n-2}\oplus \sO_{\bP^1}$.
Thus, 
either $\sO_{\bP^1}(2)\oplus\sO_{\bP^1}(1)^{\oplus n-3}\subset \sF_{|\ell}$, or 
$\sF_{|\ell}\simeq \sO_{\bP^1}(1)^{\oplus n-1}$.
So we must have $\sF=T_{X/Y}$, which  
contradicts Proposition \ref{prop:foliation_algebraically_integrable_klt_not_weak_fano}. 
\end{proof}


\begin{prop}\label{prop:ko_leaves}
Let $X$ be a normal projective variety, and
$\sF\subsetneq T_X$ an algebraically integrable $\bQ$-Fano foliation of rank $r$.
Then 
\begin{enumerate}
	\item $i_{\sF}\le r$.
	\item If $i_{\sF}= r$, then the  general log leaf of $\sF$ satisfies $(\tilde F,\tilde \Delta)\simeq (\bP^r, H)$,
		where $H$ is a hyperplane in $\bP^r$.
\end{enumerate}
\end{prop}

\begin{proof}
Let $\sL$ be an ample line bundle on $X$ such that $-K_\sF\sim_\bQ i_{\sF}c_1(\sL)$.

We denote by $\tilde e:\tilde F\to X$ the natural morphism.  
Recall from Definition \ref{defn:log_leaf} that 
$$
-(K_{\tilde F}+\tilde \Delta) \sim_\bQ  -\tilde e^*K_{\sF}\sim_\bQ i_{\sF}c_1(\tilde e^*\sL).
$$

Suppose that $i_{\sF}>r$. Then
$\tilde F\simeq\bP^r$ and $\deg(\tilde \Delta)=r+1-i$ by Theorem \ref{thm:ko}. 
In particular $(\tilde F,\tilde \Delta)$ is klt by Lemma \ref{lemma:proj_klt}. 
But this contradicts Proposition \ref{prop:foliation_algebraically_integrable_klt_not_weak_fano}.
So $i_{\sF}\le r$.

Suppose  $i_{\sF}=r$.
By Theorem \ref{thm:ko},  one of the following holds.
\begin{itemize}
	\item Either $(\tilde F,\tilde e^*\sL)\simeq  (\bP^{r},\sO_{\bP^{r}}(1))$ and $\deg(\tilde \Delta)=1$,  or 
	\item $\tilde \Delta=0$ and $(\tilde F,\tilde e^*\sL)\simeq (Q_{r},\sO_{Q_{r}}(1))$ where 
		$Q_{r}$ is a possibly singular quadric hypersurface in $\bP^{r+1}$.
\end{itemize}
In the first case, we conclude that $\tilde \Delta=H$ is a hyperplane in $\bP^r$ by Lemma \ref{lemma:proj_klt}
and Proposition \ref{prop:foliation_algebraically_integrable_klt_not_weak_fano}.
The latter case does not occur by Proposition \ref{prop:foliation_algebraically_integrable_klt_not_weak_fano}.
\end{proof}

The next lemma is useful to prove algebraic integrability of $\bQ$-Fano foliations of high index.
First we recall the notion of slope of a torsion-free sheaf.

\begin{say}
Let $X$ be an $n$-dimensional projective variety, and $\sL$ an ample line bundle on $X$. 
The \emph{slope with respect to $\sL$} of a torsion-free sheaf  $\sF$ of rank $r$ on $X$ is 
$\mu_{\sL}(\sF)=\frac{c_1(\sF)\cdot \sL^{n-1}}{r}$.
\end{say}

\begin{lemma}[{\cite[Proposition 7.5]{fano_foliation}}]\label{lemma:foliation_max}
Let $X$ be a normal projective variety, 
$\sL$ an ample line bundle on $X$,
and $\sF\subsetneq T_X$ a foliation on $X$.
Suppose that $\mu_{\sL}(\sF)>0$, and
let $C\subset X$ be a general complete intersection curve. 
Then either 
\begin{itemize}
\item $\sF$ is algebraically integrable with rationally connected general leaves, or
\item there exists an algebraically integrable subfoliation $\sG\subsetneq \sF$
with rationally connected general leaf such that 
$\det(\sG)\cdot C \ge \det(\sF)\cdot C$. 
\end{itemize}
\end{lemma}

\begin{prop}\label{prop:ko_alg_int}
Let $X$ be a normal $\bQ$-factorial projective variety,
and assume that $\Pic(X)/tors \simeq \bZ$.
Let $\sF\subsetneq T_X$
be a $\bQ$-Fano foliation of rank $r$ and index $i_{\sF}>r-1$. 
Then $\sF$ is algebraically integrable with rationally connected general leaves.
\end{prop}

\begin{proof}
Let $\sL$ be an ample line bundle on $X$ such that $-K_\sF\sim_\bQ i_{\sF}c_1(\sL)$. 
Then $\Pic(X)/tors=\bZ[\sL]$.

By Lemma \ref{lemma:foliation_max}, either $\sF$ 
is algebraically integrable with rationally connected general leaves, 
or  there is an algebraically integrable foliation
$\sG\subsetneq \sF$ of rank $s \le r-1$ and index $i_{\sG}\ge i_{\sF}$.
In the latter case, $i_{\sG}\ge i_{\sF}>r-1 \ge s$. But this contradicts
Proposition \ref{prop:ko_leaves}. Our claim follows.
\end{proof}

The following corollary is an immediate consequence of Propositions \ref{prop:ko_leaves} and \ref{prop:ko_alg_int}.

\begin{cor}\label{cor:ko_alg_int}
Let $X$ be a normal $\bQ$-factorial projective variety such that $\Pic(X)/tors \simeq \bZ$.
Let $\sF\subsetneq T_X$ be a $\bQ$-Fano foliation of rank $r$. Then 
\begin{enumerate}
	\item $i_{\sF}\le r$.
	\item If $i_{\sF}= r$, then $\sF$ is algebraically integrable, and the  general log leaf of $\sF$ satisfies $(\tilde F,\tilde \Delta)\simeq (\bP^r, H)$,
		where $H$ is a hyperplane in $\bP^r$.
\end{enumerate}
\end{cor}

In the situation of Corollary~\ref{cor:ko_alg_int}(2), we want to prove that $X$ is a normal generalized cone.
The idea is to look at  the family of leaves $\pi:U\to W$ described in Lemma~\ref{lemma:leaffoliation}.
The first step is to show that $\pi$ is a $\bP^r$-bundle.
For that we need the following characterization of projective bundles, 
which extends \cite[Lemma 2.12]{fujita85} to the singular setting.

\begin{prop}\label{prop:p_bundle}
Let $X$ be a normal variety, and $p:X \to Y$ an equidimensional  projective morphism
of relative dimension $k$ onto a normal variety. 
Suppose that there exists a $p$-ample line bundle $\sL$ on $X$ such that, for a general point $y\in Y$, 
$(X_y,\sL_{|X_y})\simeq (\bP^k,\sO_{\bP^k}(1))$. Then 
$(X,\sL)\simeq (\bP_Y(\sE),\sO_{\bP_Y(\sE)}(1))$ as varieties over $Y$, where 
$\sE:=p_*\sL$ is a vector bundle of rank $r+1$.
\end{prop}

\begin{proof}
By \cite[Proposition 3.1]{hoering}, all fibers of $p$ are irreducible and generically reduced.
Moreover, the normalization of any fiber is isomorphic to $\bP^k$, and $\sL_{|\bP^k}\simeq \sO_{\bP^k}(1)$. 
Apply \cite[Theorem 12]{kollar_husks} to conclude that $p:X \to Y$ is a 
$\bP^k$-bundle. 
\end{proof}


\begin{proof}[{Proof of Theorem \ref{thm:codimension_r_Fano_index__equal_rank_picard_number_1}}]
Let $X$ be an $n$-dimensional $\bQ$-factorial klt projective variety with $n\ge 2$, and
$\sF\subsetneq T_X$ a $\bQ$-Fano foliation of rank $r$.

\smallskip

Suppose first  $\rho(X)=1$.
By  Corollary \ref{cor:fano_foliation_picard_number_one}, 
$X$ is a klt  $\bQ$-Fano variety. By Lemma~\ref{lemma:Fano_lin_equ_and_num_equ}, 
$\Pic(X) \simeq \bZ$.
Corollary \ref{cor:ko_alg_int} then implies that  $i_{\sF}\le r$.

Suppose that $i_{\sF}= r$, and let $\sL$ be an ample line bundle on $X$ such that $-K_\sF\sim_\bQ rc_1(\sL)$. 
By Corollary \ref{cor:ko_alg_int}, $\sF$ is algebraically integrable, and the  general log leaf of $\sF$ satisfies 
$(\tilde F,\tilde \Delta)\simeq (\bP^r, H)$, where $H$ is a hyperplane in $\bP^r$.
Let $\tilde e:\tilde F\to X$ denote the natural morphism. 
Then $\tilde e^*\sL\simeq\sO_{\bP^r}(1)$.

Let $W$ be the normalization of the closure in $\Chow(X)$ of the subvariety parametrizing 
general leaves of $\sF$, and  $U$ the normalization of the universal cycle over $W$,
with universal family morphisms:

\centerline{
\xymatrix{
U \ar[r]^{e}\ar[d]_{\pi} & X \ . \\
 W &
}
}

\noindent The line bundle $e^*\sL$ is $\pi$-ample and restricts to $\sO_{\bP^r}(1)$ on a general fiber $U_w\simeq \bP^r$ of $\pi$.
By Proposition \ref{prop:p_bundle},  $\pi$ is a $\bP^r$-bundle. 
More precisely, setting $\sE:=\pi_*e ^*\sL$, $U$ is isormorphic to 
$\mathbb{P}_W(\sE)$ over $W$, and under this isomorphism  $e^*\sL$ 
corresponds to the tautological line bundle $\sO_{\mathbb{P}_W(\sE)}(1)$.

We will show that the above diagram realizes $X$ as a normal generalized cone over $\big(W, \det(\sE)\big)$. 

By Remark \ref{remark:definition_log_leaf}, there is a canonically defined $\bQ$-Weil divisor $B$ on $U$ such that 
$K_{U/W}+B\sim_\bQ e^*K_\sF$. Notice that since 
$K_{U/W}$ and $e^*K_\sF$ are $\mathbb{Q}$-Cartier divisors, hence so is $B$.
Since $\tilde \Delta$ is integral, the support of $B$ has a unique irreducible component $B'$  
that dominates $W$. Moreover, $B'$ appears with coefficient $1$ in $B$. 
Let $E$ be the exceptional locus of $e$. Note that $E$ has pure codimension 1 since $X$ is $\bQ$-factorial. 
We shall show that $B=B'=E$.

Let $m$ be a positive integer such that $mB$ is integral.
Recall that 
$\sO_U(K_{U/W})\simeq \sO_{\mathbb{P}_W(\sE)}(-r-1)\otimes \pi^*\det(\sE)$, and so
$m\,B\in |\sO_U(m)\otimes \pi^*\det(\sE)^{\otimes -m}|$.
Let $C$ be a smooth complete curve, and $C \to W$  a non-constant morphism
whose image is not contained in the image of any irreducible component of 
$\textup{Supp}(B)$ distinct from $B'$. 
Let $U_C$ be the normalization of $C \times _W U$, with induced
morphisms $\pi_C:U_C\to C$ and $e_C:U_C\to X$, and 
set $\sE_C:={\pi_C}_*{e_C}^*{\sL}$.
Then $\sE_C$ is nef, 
$-(K_{U_C/C}+B_{|U_C})\sim_\bQ {e_C}^*(-K_{\sF})$, and 
$m\,B_{|U_C}\in |\sO_{U_C}(m)\otimes \pi_C^*\det(\sE_C)^{\otimes -m}|$.
Suppose that $B_{|U_C}$ is reducible, and write
$B_{|U_C}=B'_{|U_C}+\sum_{1\le i\le k}a_i F_i$, where the $F_i$'s are fibers of $\pi_C$ and the  $a_i$'s are positive rational numbers.
Let $t \gg0$ be such that $t\,a_i\in\bZ$ for all $1 \le i\le k$.
Then 
$mt\,B'_{|U_C}\in |\sO_{U_C}(mt)\otimes \pi_C^*\det(\sE_C)^{\otimes -mt}\otimes\sO_{U_C}(-mt\sum_{1\le i\le k}a_i F_i)|$, and  
thus 
$h^0(C,S^{mt}\sE_C\otimes\det(\sE_C)^{\otimes -mt}\otimes\sO_C(-mt\sum_{1\le i\le k}a_i\,c_i))\neq 0$, 
where $c_i=\pi_C(F_i)\in C$. 
But this contradicts Lemma \ref{lemma:vanishing} below.
We conclude that $B_{|U_C}$ is  irreducible,  and thus $B=B'$.
To show that $E=B$,  
let $\tilde C\subset U_C$ be a curve such that $\pi_C(\tilde C)=C$, and suppose that 
$e_C$ maps $\tilde C$ to a point.
Then 
$B \cdot \tilde C = -K_{U_C/C} \cdot \tilde C
= -\pi_C^*\det(\sE_C) \cdot \tilde C$.
Since $\sL$ is ample, the later is $<0$ again by Lemma~\ref{lemma:vanishing}. 
Hence, $ \tilde C \subset \textup{Supp}(B)$.
We conclude that $E=B$. 
This argument also shows that
$B \to W$ is a $\bP^{r-1}$-bundle.

Next we observe that $W$ is $\bQ$-factorial with Picard number $\rho(W)=1$.
Indeed, since $E$ is irreducible and $X$ is $\bQ$-factorial with Picard number $\rho(X)=1$, 
we have $\dim(\textup{Pic}(U)\otimes\bQ)\le \dim(\textup{Cl}(U)\otimes\bQ)=2$.
On the other hand, $e$ is not finite, and thus $\dim(\textup{Pic}(U)\otimes\bQ)\ge 2$.
Our claim follows.

Set $\sG:={(\pi_{|B})}_*{(e^*\sL)}_{|B}$, and notice that it is a nef vector bundle on $W$.
The inclusion $B\subset U$ corresponds to $\sE\twoheadrightarrow \sG$.
Since $B\in |\sO_U(1)\otimes \pi^*\det(\sE)^*|$,
we have an exact sequence
$$0\to \det(\sE) \to \sE \to \sG \to 0,$$
and $\det(\sG)\simeq \sO_W$.
We claim that $\sG\simeq\sO_W^{\oplus r}$.
By replacing $W$ with a resolution of singularities, 
and applying the projection formula, 
we may assume that $W$ is smooth.

By Corollary \ref{cor:fano_foliation_picard_number_one}, 
$X$ is a $\bQ$-Fano variety, and hence 
rationally connected by Theorem~\ref{thm:rc}.
Thus $W$ is rationally  connected, hence simply connected 
by \cite[Corollary 4.18]{debarre}.
Our claim then  follows  from \cite[Theorem 1.18]{demailly_peternell_schneider94}.

Since $W$ is a $\bQ$-Fano variety
with klt singularities and Picard number $\rho(W)=1$,
$\det(\sE)$ is ample, and 
$h^1(W,\det(\sE))=0$ (see \cite[Theorem 1.2.5]{kmm}).
Hence $\sE\simeq \det(\sE)\oplus \sO_W^{\oplus r}$, and
$X$ is a normal generalized cone over
$(W,\sM)$, where $\sM=\det(\sE)$.

For the last statement, observe that 
the quotient $\sE\twoheadrightarrow \sM$ corresponds to a section $\sigma$ of $\pi$, contained in $U \setminus E\simeq X\setminus e(E)\subset X$.
Moreover,  $\sL_{|\sigma(W)}\simeq\sM$ and $N_{\sigma(W)/X} \simeq \sM^{\oplus r}$. Thus 
$\omega_{\sigma(W)}\simeq {\omega_X}_{|\sigma(W)}\otimes 
\sM^{\otimes r}$, and $i_W=k(i_X-r)\ge 1$ for some positive integer $k$. This completes the proof.

\smallskip

Suppose now that $\sF$ has rank $r=n-1$.
Then apply Proposition \ref{prop:ko_picard_number_at_least_2} and reduce to the previous case. 
\end{proof}

\begin{lemma}\label{lemma:vanishing}
Let $C$ be a smooth complete curve, and $\sG$ a vector bundle of rank $r\ge 1$ on $C$.

\begin{enumerate}
\item If $\sG$ is nef and $\deg(\sG)=0$, then $h^0(C,\sG)\le r$.
\item If $\sG$ is nef, then $h^0(C,S^{m}(\sG)\otimes\det(\sG)^{\otimes -m}\otimes \sM^{\otimes -1})=0$ for any integer $m\ge 1$, 
and any line bundle $\sM$ on $C$ with $\deg(\sM)>0$.
\item Let $\sO_{\bP}(1)$ denote the tautological line bundle on $\mathbb{P}_C(\sG)$.
If $\sO_{\bP}(1)$ is nef and big, then $\deg(\sG)>0$.
\end{enumerate}
\end{lemma}

\begin{proof}

We prove (1) by induction on $r$. The result is clear if $r=1$. 
So we assume that  $r\ge 2$. We may assume that there exists a nonzero section $s\in H^0(C,\sG)$,
which induces a nonzero map  $\sO_C\hookrightarrow \sG$.
Since $\sG$ is nef and $\deg(\sG)=0$, $\sO_C$ is in fact a subbundle of $\sG$.
So $\sQ=\sG/\sO_C$ is a nef vector bundle on $C$, and $\deg(\sQ)=0$. 
By induction, 
$h^0(C,\sQ)\le r-1$. Our claim follows.

To prove (2), 
we argue by contradiction. Suppose that
$h^0(C,S^{m}(\sG)\otimes\det(\sG)^{\otimes -m}\otimes\sM^{\otimes -1})\neq 0$ for some $m\ge 1$, and 
some line bundle $\sM$ on $C$ with $\deg(\sM)>0$.
Then there is a nonzero map $\det(\sG)^{\otimes m}\otimes\sM \hookrightarrow S^{m}(\sG)$. 
Since $\mu(\det(\sG)^{\otimes m}\otimes\sM)=mr \mu(\sG)+\mu(\sM)> m \mu(\sG)=\mu(S^{m}(\sG))$, 
$\sG$ is not semistable, and 
there exists a subbundle $\sH\subset \sG$ of rank $k\geq 1$ such that 
$\mu(\sH) \ge r\mu(\sG)+\frac{1}{m}\mu(\sM)$. 
The quotient $\sQ:=\sG/\sH$ is a nef vector bundle, and thus
$$\deg(\sG)=\deg(\sH)+\deg(\sQ)\ge \deg(\sH)\ge k \deg(\sG)+\frac{k}{m}\deg(\sM)\ge \deg(\sG)+\frac{k}{m}\deg(\sM),$$
a contradiction.

Assertion (3) is an immediate consequence of (1).
\end{proof}


\section{Codimension 1 del Pezzo foliations on varieties with mild singularities}\label{section:dp}

In this section we prove Theorem \ref{thm:classification_del_pezzo}.
The proof naturally splits in two parts: the case of Picard number $1$, and
Picard number $>1$. 
First we treat the case of Picard number $1$. 
We start by proving algebraic integrability.

\begin{prop}\label{prop:codimension_one_del_pezzo_algebraic_leaves}
Let $X$ be a $\bQ$-factorial klt projective variety
such that $\rho(X)=1$, and
let $\sF\subsetneq T_X$
be a del Pezzo foliation.
Then either 
\begin{enumerate}
\item $\sF$ is algebraically integrable with rationally connected general leaves, or
\item $X$ is a generalized cone over a projective normal variety $Z$ with $\rho(Z)=1$, 
and  $\sF$ is the pullback by $X\dashrightarrow Z$
of a foliation induced by a nonzero global section of $T_Z$.
Moreover, $Z$ has $\bQ$-factorial klt singularities,  and
$i_Z=k(i_X-r+1)>0$ for some positive integer $k$.
\end{enumerate} 
\end{prop}

\begin{proof}
Let $\sL$ be an ample line bundle on $X$ such that $-K_\sF\sim_\bQ (r-1)c_1(\sL)$.
Then $\textup{Pic}(X)=\bZ[\sL]$ by Lemma \ref{lemma:Fano_lin_equ_and_num_equ}.

Suppose that $\sF$ is not algebraically integrable with rationally connected general leaves.
By Lemma \ref{lemma:foliation_max}, there exists a foliation with algebraic leaves
$\sG\subsetneq \sF$ of rank $s \le r-1$ and index $i_{\sG}\ge r-1$.
By Proposition~\ref{prop:ko_leaves}, we must have $s=i_{\sG}=r-1$.
By Theorem \ref{thm:codimension_r_Fano_index__equal_rank_picard_number_1},
$X$ is a generalized cone as described in (2).
The fact that $\sF$ is the pullback by $X\dashrightarrow Z$
of a foliation induced by a nonzero global section of $T_Z$ follows from \cite[Lemma 6.7]{fano_foliation}.
\end{proof}

\begin{thm}\label{thm:codimension_1_del_pezzo_picard_number_1}
Let $X$ be a factorial klt projective variety of dimension $n\ge 3$ and Picard number $\rho(X)=1$.
Suppose that $X$ admits a codimension $1$ del Pezzo foliation $\sF\subsetneq T_X$.
Then one of the following holds.
\begin{enumerate}
\item $\sF$ is a degree $1$ foliation on $X\simeq\bP^n$.
\item $X$ is isomorphic to a (possibly singular) quadric hypersurface $Q_n\subset \bP^{n+1}$, and $\sF$ is a pencil of hyperplane sections of $Q_n$.
\item $X$ is a normal generalized cone over a projective normal surface $Z$ with $\rho(Z)=1$, and 
$\sF$ is the pullback by $X\dashrightarrow Z$
of a foliation induced by a nonzero global section of $T_Z$.
Moreover, $Z$ is factorial and klt, $i_X>n-2$ and
$i_Z=k(i_X -n+2)$ for some positive integer $k$.
\end{enumerate}
\end{thm}

The proof of Theorem~\ref{thm:codimension_1_del_pezzo_picard_number_1} follows  the line of argumentation of \cite{lpt3fold}.
We will use the following results, which follow from [{\it loc.cit.}]. We provide proofs for the reader's convenience.

\begin{lemma}[see {\cite[Lemma 3.1]{lpt3fold} or \cite{ghys00}}]\label{lemma:invariant_divisor_zero_first_chern_class}
Let 
$X$ be a smooth variety of dimension $\ge 3$, and  $\sF\subsetneq T_X$  a codimension one foliation on $X$. Suppose that there exists
a nonzero $\bQ$-divisor $D$ on $X$ such that $D\sim_{\bQ} 0$
and such that $\textup{Supp}(D)$ is invariant by $\sF$.
Then either $\sF$ is defined by a logarithmic $1$-form with poles exactly along $\textup{Supp}(D)$, or
there exists a codimension 2 foliation $\sG\subsetneq \sF$ with 
$K_\sG\le K_\sF$.
\end{lemma}

\begin{proof}
Let $\omega\in H^0(X,\Omega^1_X\otimes N_\sF)$ be a twisted $1$-form defining $\sF$. We may assume that $D$ is integral, and $D \sim_\bZ 0$. 
Let $\{D_i\}_{i\in I}$ be the set of irreducible components of $D$, and write
$D=\sum_{i\in I}d_i D_i$. 
There exist an affine open cover $(U_\alpha)_{\alpha\in A}$ of $X$, a nonzero rational function $f$ on $X$, and regular functions $(f_{i\alpha})_{i\in I,\alpha\in A}$ such that, over $U_\alpha$,
$f_{i\alpha}$ is a defining equation for $D_i$, and
$$\prod_{i\in I}f_{i\alpha}^{d_i}=u_\alpha f,$$
where $u_\alpha$ is a unit on $U_\alpha$. Thus, for $\alpha, \beta\in A$,
$$\sum_{i\in I}d_i(\frac{df_{i\alpha}}{f_{i\alpha}}-\frac{df_{i\beta}}{f_{i\beta}})=\frac{du_\alpha}{u_\alpha}-\frac{du_\beta}{u_\beta}$$
over $U_{\alpha\beta}$.
For each $\alpha \in A$,
set $\xi_\alpha=\sum_{i\in I}d_i\frac{df_{i\alpha}}{f_{i\alpha}}-\frac{du_\alpha}{u_\alpha}$. 
Then the $\xi_\alpha$'s define a logarithmic $1$-form $\xi$ on $X$ with poles exactly along $\textup{Supp}(D)$, and 
$\omega\wedge\xi\in H^0(X,\Omega^2_X\otimes N_\sF)$
since $\textup{Supp}(D)$ is invariant by $\sF$. If $\omega\wedge\xi=0$ then 
$\sF$ is defined by the logarithmic $1$-form $\xi$.
If $\omega\wedge\xi\neq 0$, then 
$\omega\wedge\xi$ defines
a codimension 2 foliation $\sG\subsetneq \sF$ such that
$K_\sG\le K_\sF$. This completes the proof.
\end{proof}

\begin{lemma}\label{lemma:codimension_one_first_integral}
Let $X$ be a $\bQ$-factorial normal projective variety  with Picard number $\rho(X)=1$, and 
$\sF\subsetneq T_X$ an algebraically integrable 
codimension one foliation with a rational first integral $\pi:X \dashrightarrow C$. 
Then either any fiber of $\pi$ is irreducible
or there exists a codimension 2 foliation $\sG\subsetneq \sF$ on $X$ with $K_\sG\le K_\sF$.
\end{lemma}

\begin{proof}Suppose that
$\pi$ has a reducible fibre. Let $D_1$ be an irreducible component of a reducible fiber, and let $D_2$ be a general fiber. Since $X$ is $\bQ$-factorial with 
$\rho(X)=1$, $D_1\sim_{\bQ} d\,D_2$ for some nonzero rational number $d$. Set $D:=D_1-d\,D_2\sim_\bQ 0$. By  
Lemma \ref{lemma:invariant_divisor_zero_first_chern_class} above applied to the smooth locus of $X$, 
either $\sF$ is defined by a logarithmic $1$-form with poles exactly along $\textup{Supp}(D)$, or
there exists a codimension 2 foliation $\sG\subsetneq \sF$ with 
$K_\sG\le K_\sF$.

Suppose that $\sF$ is defined by a logarithmic $1$-form $\xi$ with poles exactly along $\textup{Supp}(D)$. 
Since a general fiber of $\pi$ is irreducible, $\xi=\pi^*\gamma$, where $\gamma$ is a rational $1$-form on $C$. This implies that $\xi$ has poles along fibers, a contradiction.
\end{proof}

\begin{proof}[Proof of Theorem~\ref{thm:codimension_1_del_pezzo_picard_number_1}]

Let $\sL$ be an ample line bundle on $X$ such that $-K_\sF\sim_\bQ (n-2)c_1(\sL)$. 
Then $\textup{Pic}(X)=\bZ[\sL]$ by  Lemma~\ref{lemma:Fano_lin_equ_and_num_equ}.

If $\sF$ is not algebraically integrable, then the result follows from Proposition \ref{prop:codimension_one_del_pezzo_algebraic_leaves}.

So from now on we assume that $\sF$ is algebraically integrable, and consider a rational first integral 
$\pi:X\dashrightarrow W$. 
By Corollary \ref{cor:fano_foliation_picard_number_one}, 
$X$ is a $\bQ$-Fano variety, and hence 
rationally connected by Theorem~\ref{thm:rc}.
So we must have $W\simeq\bP^1$.

By Lemma \ref{lemma:codimension_one_first_integral},
either any fiber of $\pi$ is irreducible,
or there exists a codimension 2 foliation $\sG\subsetneq \sF$ on $X$ with $K_\sG\le K_\sF$.
In the latter case, 
$\sG$ has index $i_\sG \ge n-2$, and so, 
by Theorem~\ref{thm:codimension_r_Fano_index__equal_rank_picard_number_1}, 
we are in case (3).

From now on we assume that any fiber of $\pi$ is irreducible.
We denote by $s\ge 0$ the number of multiple fibers of $\pi$. 
When $s\ge 1$, we denote by $m_1F_1,\ldots, m_sF_s$ the non reduced fibers of $\pi$, with 
$F_i$ reduced and $m_1\ge \cdots \ge m_s \ge 2$.
For notational convenience, we let $F_{s+1}$ and $F_{s+2}$ be general fibers of $\pi$,
and set $m_{s+1}=m_{s+2}=1$.
For $1\le i\le s+2$, set $p_i:=\pi(F_i)\in \bP^1$, and let  $k_i$ be the positive integer
such that $\sO_X(F_i)\simeq\sL^{\otimes k_i}$.
Notice that  
\begin{equation}\label{multiplicity}
k_1m_1=\cdots=k_sm_s=k_{s+1}m_{s+1}=k_{s+2}m_{s+2}.
\end{equation}
In particular, $k_1\le  \cdots  \le k_{s+1} = k_{s+2}$.

Let $R(\pi)=\sum_{1\le i\le s}\frac{m_i-1}{m_i}\pi^*\sO_{\bP^1}(p_i)$
be the ramification divisor of $\pi$. By \cite[Lemma 4.4]{druel99}, we have 
$$
\det(N_\sF)\simeq \pi^*\sO_{\bP^1}(2)\otimes\sO_X\big(-R(\pi)\big)
\simeq\sO_X\big(F_1+F_2-\sum_{3\le i\le s}(m_i-1)F_i\big)
\simeq \sL^{\otimes \big(k_1+k_2-\sum_{3\le i\le s}k_i(m_i-1)\big)}.
$$
On the other hand, 
$$
\sO_X(K_X)\simeq \sO_X(K_\sF)\otimes \det(N^*_\sF).
$$ 
Thus the index $i_X$ of $X$ satisfies 
\begin{equation}\label{canonical_bundle_formula}
i_X=n-2+k_1+k_2-\sum_{3\le i\le s}k_i(m_i-1).
\end{equation}

By Corollary \ref{cor:index}, we must have $i_X \ge n-1$.
If $i_X \ge n+1$, then, by Theorem \ref{thm:ko}, $i = n+1$,
$(X,\sL) \simeq (\bP^n,\sO_{\bP^n}(1))$
and $\sF$ is degree one foliation on $\bP^n$.
So from now on we assume that $n-1\le i_X\le n$.

We claim that there exists 
a finite cover $c:C \to \bP^1$, unramified over 
$\bP^1\setminus \{p_1,\ldots,p_s\}$, and having 
ramification index $m_i$ at each point of $c^{-1}(p_i)$.
If $s \le 2$, then $i_X=n$, $k_1=k_2=1$,
and $m_1=m_2$ 
by \eqref{canonical_bundle_formula} and \eqref{multiplicity}.
In particular, either $s=0$ (when $m_1=m_2=1$) or $s=2$ (when $m_1=m_2\ge 2$).  
In either case the existence of 
$c:\bP^1\to\bP^1$ is clear. 
If $s>2$, the existence of $c:C \to \bP^1$ follows from 
\cite[Lemma 6.1]{kobayashi_ochiai82}.
Let $X'$ be the normalization of $X$ in the function field of
$X\times_{\bP^1} C$, with induced morphism $c' : X'\to X$, and set $\sL':={c'}^*\sL$.
An easy calculation shows that $c'$ is \'etale in codimension one. Hence $c'$ is \'etale 
over $X\setminus \textup{Sing}(X)$ by purity of the branch locus. 
In particular, $X\setminus\textup{Sing}(X)$ is not simply connected if $s>0$.

Suppose that $i_X = n$. Then, by Theorem \ref{thm:ko},
$(X,\sL)\simeq (Q_n,\sO_{Q_n}(1))$ where 
$Q_n$ is a possibly singular quadric hypersurface in $\bP^{n+1}$. 
Moreover, since $X\setminus \textup{Sing}(X)$ is simply connected, 
$s=0$ and $k_1=k_2=1$ by \eqref{canonical_bundle_formula}. In other words,
$\sF$ is a pencil of hyperplane sections of $Q_n$.

Suppose now that $i_X=n-1$. In particular, $(X,\sL)$ is a (normal) del Pezzo variety. 
Then, by (\ref{canonical_bundle_formula}) and (\ref{multiplicity}), we must have $s=3$ and either 
$(k_1,k_2,k_3,m_1,m_2,m_3)=(1,k,k,2k,2,2)$ for some integer $k\ge 1$, or
$(k_1,k_2,k_3,m_1,m_2,m_3)\in \big\{(2,2,3,3,3,2), (3,4,6,4,3,2), (6,10,15,5,3,2)\big\}$.


Note that $K_{X'}={c'}^*K_X$ since
$c':X'\to X$ is \'etale in codimension one. This implies
that $X'$ has Gorenstein klt singularities (see
\cite[Proposition 5.20]{kollar_mori}). Thus $(X',\sL')$ is a (normal) del Pezzo variety, and hence 
rationally connected by Theorem~\ref{thm:rc}.
So $C\simeq \bP^1$. By Hurwitz formula, either  $\deg(c)=4k$ for some integer $k\ge 1$, or $\deg(c)\in \{12, 24, 60\}$, respectively. 
Since $\sL^n\ge 3$ by  Lemma \ref{lemma:vanishing_del_pezzo}, we have 
${\sL'}^n=\deg(c)\cdot \sL^n\ge 12$. 
On the other hand, Proposition \ref{prop:del_pezzo_degree} tells us that 
${\sL'}^n\le 9$, yielding a contradiction and completing the proof.
\end{proof}

Now we consider the case of Picard number $\rho(X) \ge 2$.

\begin{thm}\label{thm:classification_picard_number_at_least_2}
Let $X$ be a factorial klt projective variety of dimension $n\ge 3$ and Picard number $\rho(X)\ge 2$.
Suppose that $X$ admits a codimension $1$ del Pezzo foliation $\sF\subsetneq T_X$.
Then  there is an exact sequence of vector bundles 
$0\to \sK\to \sE\to \sV \to 0$ on $\bP^1$ such that $X\simeq \p_{\p^1}(\sE)$, and $\sF$ is the pullback via the relative linear projection $X\map Z=\p_{\p^1}(\sK)$ of a foliation 
on $Z$  induced by a nonzero global section of $T_Z\otimes q^*\det(\sV)^*$.
Here $q:Z\to \p^1$ denotes the natural projection.  
Moreover,  one of the following holds.
\begin{enumerate}
\item 
$(\sE,\sK)\simeq \big(\sO_{\bP^1}(2)\oplus \sO_{\bP^1}(a)^{\oplus 2},\sO_{\bP^1}(a)^{\oplus 2}\big)$
for some positive integer $a$.
\item 
$(\sE,\sK)\simeq \big(\sO_{\bP^1}(1)^{\oplus 2}\oplus \sO_{\bP^1}(a)^{\oplus 2},\sO_{\bP^1}(a)^{\oplus 2}\big)$
for some positive integer $a$.
\item 
$(\sE,\sK)\simeq\big(\sO_{\bP^1}(1)\oplus \sO_{\bP^1}(a)\oplus \sO_{\bP^1}(b),\sO_{\bP^1}(a)\oplus \sO_{\bP^1}(b)\big)$ for distinct 
positive integers $a$ and $b$.
\end{enumerate}
\end{thm}

\begin{proof}
Note that $X$ has Gorenstein rational singularities.
Denote by $n$ the dimension of $X$, and 
let $\sL$ be an ample line bundle on $X$ such that $-K_\sF\sim_\bQ (n-2)c_1(\sL)$. 
By Theorem \ref{thm:main_foliation}, $K_X+(n-2)c_1(\sL)$ is not nef. 
By \cite[Theorem 3]{fujita85}, one of the following holds: 

\begin{enumerate}

\item [(A)] $X$ is the blowup of $Y$ at a smooth point (see also \cite[Theorem 3.1]{andreatta}).
\item [(B)] There exists a surjective morphism $p:X\to Y$ onto a normal projective variety $Y$ with Picard number $\rho(Y)=\rho(X)-1$, 
with general fiber $F$, satisfying one of the following conditions.
\begin{enumerate}
\item [(1)] $\dim(Y)=1$ and $(F,\sL_{|F})\simeq (\bP^{2},\sO_{\bP^{2}}(2))$.
\item [(2)] $\dim(Y)=1$ and $(F,\sL_{|F})
\simeq (Q_{n-1},\sO_{Q_{n-1}}(1))$,  where $Q_{n-1}$ is a (possibly singular) quadric hypersurface in $\bP^n$.
\item [(3)] $\dim(Y)=1$, $(F,\sL_{|F})\simeq (\bP^{n-1},\sO_{\bP^{n-1}}(1))$,
and there exists an ample vector bundle $\sG$ on $Y$ such that $X\simeq\bP_Y(\sG)$ as varieties over $Y$, with tautological line bundle $\sO_X(1)\simeq\sL$.
\item [(4)] $\dim(Y)=2$ and $(F,\sL_{|F})\simeq (\bP^{n-2},\sO_{\bP^{n-2}}(1))$.
\end{enumerate}
\end{enumerate}

First we consider case (B). 
We claim that $T_{X/Y} \nsubseteq \sF$.
Since $F$ is log terminal, $\sF\neq T_{X/Y}$ by Proposition \ref{prop:foliation_algebraically_integrable_klt_not_weak_fano}.
Hence, if $T_{X/Y} \subset \sF$, we must be in case (B4).
In this case, by \cite[Lemma 6.7]{fano_foliation}, $\sF$ is the pullback by $p$ of a rank 1 foliation on $Y$, and thus 
$\sF_{|\ell}\simeq T_{F|\ell}\oplus \sO_{\bP^{1}}$, where $\ell$ is a general line on $F$.
But this contradicts the assumptions that $-K_\sF\sim_\bQ (n-2)c_1(\sL)$, and proves the claim.

By considering the possible splitting types of the restriction of $\sF$ to a  general line on $F$, 
and using the fact that $T_{X/Y} \nsubseteq \sF$, we see that cases (B1) and (B4)
do not occur.

Next we show that case (B2) cannot occur either. 
Let $C\subset Q_{n-1}\setminus \textup{Sing}(Q_{n-1})$ be a general conic. 
Then $X$ is smooth in a neighborhood of $C$ and 
$\sF_{|C}\subset {T_X}_{|C}\simeq \sO_{\bP^1}(2)^{\oplus n-1}
\oplus\sO_{\bP^1}$.
Since $\sF\neq T_{X/Y}$, 
the foliation $\sG:=\sF\cap T_{X/Y}$ has rank $n-2$, and 
$\sG_{|C}\simeq\sO_{\bP^1}(2)^{\oplus n-2}$. 
It induces
a foliation on $F\simeq Q_{n-1}$ defined by a twisted $1$-form 
in $H^0(F,\Omega^{[1]}_F\otimes\sL_{|F})$. But $h^0(F,\Omega^{[1]}_F\otimes\sL_{|F})=0$, a contradiction.

Suppose we are in case (B3). The short exact sequence
$$0\to T_{X/Y} \to T_{X} \to {p}^*T_{Y} \to 0$$
yields an exact sequence
$$0\to \wedge^{n-1}T_{X/Y} \to \wedge^{n-1}T_{X} 
\to \wedge^{n-2}T_{X/Y}\otimes {p}^*T_{Y} \to 0.$$
Since $\sF\nsubseteq T_{X/Y}$, 
$h^0(X,\wedge^{n-2}T_{X/Y}\otimes {p}^*T_{Y}\otimes \sL^{\otimes -(n-2)})\neq 0$.
By \cite[Proposition 9.11 (4)]{fano_foliation},
$Y\simeq\bP^{1}$.
The description of $X$ and $\sF$ now follows from 
\cite[Theorem 9.6]{fano_foliation}.

\medskip

Finally, we show that case (A) does not occur.
Let $p:X\to Y$ be the blowup of $Y$ at a smooth point $y\in Y$, with exceptional divisor $E$.
Then $\sL+E\simeq p^*\sM$ for some ample line bundle $\sM$ on $Y$, and $\sF$ induces a codimension $1$ del Pezzo foliation $\sG$ on $Y$ with 
$\det(\sG)\simeq\sM^{\otimes n-2}$. 

By replacing $(X,\sF,\sL)$ 
with $(Y,\sG,\sM)$ successively, 
we may assume that 
either $\rho(Y)=1$ or $Y$ satisfies (B3).
In either case there is an unsplit covering family $H$ of rational curves on $Y$
such that $\sM\cdot H=1$ 
(when $\rho(Y)=1$, this follows from the classification in Theorem \ref{thm:codimension_1_del_pezzo_picard_number_1}).
Let $C$ be a curve from the family $H$ passing through $y$, and
$\tilde C$ its strict transform in $X$.
Then 
$$\sM\cdot H=\sM\cdot C=\pi^*\sM\cdot \tilde C=\sL\cdot \tilde C+E\cdot \tilde C\ge 2,$$
contradicting the choice of $H$. This completes the proof.
\end{proof}

\begin{rem}
Under the assumptions of Theorem~\ref{thm:classification_picard_number_at_least_2},
we have shown in particular that $X$ is smooth.
So $\sF$ is algebraically integrable, and \cite[Remark 7.11]{fano_foliation} provides complete description of the general log leaf. 
\end{rem}

\begin{proof}[Proof of Theorem~\ref{thm:classification_del_pezzo}]
Apply Theorems~\ref{thm:codimension_1_del_pezzo_picard_number_1} and \ref{thm:classification_picard_number_at_least_2}.
\end{proof}


\section{Regular foliations}\label{section:regular}

In this section, we discuss the singular locus of $\bQ$-Fano foliations
on varieties with mild singularities. 
Let $\sF\subsetneq T_X$ be a regular foliation on a complex projective manifold $X$. If $\sF$ is algebraically integrable, then $-K_\sF$ is not ample by Proposition \ref{prop:foliation_algebraically_integrable_klt_not_weak_fano}. We start by generalizing this result
to arbitrary regular foliations on complex projective manifolds with Picard number one.

\begin{thm}Let $\sF\subsetneq T_X$ be a regular foliation on a complex projective manifold $X$ with Picard number $\rho(X)=1$. 
Then $-K_\sF$ is not ample.
\end{thm}

\begin{proof}
Suppose to the contrary that $-K_\sF$ is ample.
Set $\sQ:=T_X/\sF$, $\sL:=\det(\sF)$, and denote by $r$ the rank of $\sF$. Then 
$\det(\sQ)\simeq \sO_X(-K_X)\otimes \sL^{\otimes -1}$.
By \cite[Corollary 3.4]{baum_bott70},  
$\det(\sQ)^{\dim(X)}=0$. Since $\rho(X)=1$, we must have
$\det(\sQ)\equiv 0$. This implies that $X$ is a Fano manifold, and, by Lemma \ref{lemma:Fano_lin_equ_and_num_equ}, 
$\det(\sQ)\simeq \sO_X$.
So $h^0(X,\Omega^{\dim(X)-r}_X)\neq 0$.
On the other hand, by Hodge symmetry, 
$h^0(X,\Omega^{\dim(X)-r}_X)
=h^{\dim(X)-r}(X,\sO_X)$, and the latter vanishes by Kodaira vanishing theorem, a contradiction.
\end{proof}

\begin{say}[The Atiyah class of a locally free sheaf]  \label{say:atiyah}
Let $X$ be a smooth variety, and $\sE$ a locally free sheaf of rank $r\ge 1$ on $X$.
Let $J_X^1(\sE)$ be the sheaf of $1$-jets of $\sE$. 
I.e., as a sheaf of abelian groups on $X$, 
$J_X^1(\sE)\simeq \sE\oplus (\Omega_X^1\otimes\sE)$,  and the $\sO_X$-module structure is given 
by $f(e,\alpha)=(fe,f\alpha-df\otimes e)$, where $f$, $e$ and $\alpha$ are 
local sections of $\sO_X$, $\sE$ and $\Omega_X^1\otimes\sE$, respectively. 
The \emph{Atiyah class} of $\sE$ is defined to be the element 
$at(\sE)\in H^1(X,\Omega_X^1\otimes\sE\hspace{-0.07cm}\textit{nd}_{\sO_X}(\sE))$ 
corresponding to the Atiyah extension 
$$
0\to \Omega_X^1\otimes\sE \to J_X^1(\sE) \to \sE \to 0.
$$
It can be explicitly described as follows. 
Choose an affine open cover $(U_i)_{i\in I}$ of $X$ such that $\sE$ admits a frame
$f_i:\sO_{U_i}^{r} \overset{\sim}{\to}\sE|_{U_i}$ for each $U_i$. 
For $i,j\in I$, define 
$f_{ij}:={f_j^{-1}}|_{U_{ij}}\circ {f_i}|_{U_{ij}}$. Then 
$$
at(\sE)=\big[(-{f_j}|_{U_{ij}}\circ {df_{ij}}|_{U_{ij}}\circ {f_i^{-1}}|_{U_{ij}})_{i,j}\big]\in 
H^1(X,\Omega_X^1\otimes\sE\hspace{-0.07cm}\textit{nd}_{\sO_X}(\sE))
$$ 
(see \cite[Proof of Theorem 5]{atiyah57}).
\end{say}

The next result follows from the proof \cite[Proposition 3.3]{baum_bott70} and \cite[Corollary 3.4]{baum_bott70}. 
We sketch the proof for the reader's convenience.

\begin{lemma}\label{lemma:atiyah_class_foliation}
Let $X$ be a smooth variety, and $\sF\subsetneq T_X$ a regular foliation. Set $\sQ=T_X/\sF$.
Then $at(\sQ)\in H^1(X,\Omega^1_X\otimes\sE\hspace{-0.07cm}\textit{nd}_{\sO_X}(\sQ))$ is in the image of
the natural map 
$$H^1(X,\sQ^*\otimes\sE\hspace{-0.07cm}\textit{nd}_{\sO_X}(\sQ))\to
H^1(X,\Omega^1_X\otimes\sE\hspace{-0.07cm}\textit{nd}_{\sO_X}(\sQ)).$$
\end{lemma}

\begin{proof}
Since $\sF$ is regular, there is an exact sequence
$$0\to \sQ^* \to \Omega_X^1 \to \sF^*\to 0.$$
Taking cohomology yields 
$$H^1(X,\sQ^*\otimes\sE\hspace{-0.07cm}\textit{nd}_{\sO_X}(\sQ))\to
H^1(X,\Omega^1_X\otimes\sE\hspace{-0.07cm}\textit{nd}_{\sO_X}(\sQ))
\overset{\delta}{\to}
H^1(X,\sF^*\otimes\sE\hspace{-0.07cm}\textit{nd}_{\sO_X}(\sQ))
.$$
We must show that $\delta(at(\sQ))=0$.

Denote by $q$  the rank of $\sQ$.
Choose an affine open cover $(U_i)_{i\in I}$ of $X$ such that, over each $U_i$, 
$\sQ$ admits a frame 
$\alpha_i:\sO_{U_i}^{q} \overset{\sim}{\to}\sQ^*|_{U_i}$.
By assumption, $\sF$ is stable under the Lie bracket.
This is equivalent to saying that $d\sQ^*\subset \sQ^*\wedge \Omega_X^1$. 
Thus, viewing $\alpha_i$ as a line vector whose entries are local sections of $\sQ^*\subset\Omega_X^1$  over $U_i$,
we get a matrix $\beta_i$,
whose entries are local sections of $\sQ^*\subset\Omega_X^1$ over $U_i$, such that $d\alpha_i=\alpha_i \wedge \beta_i$.

For $i,j\in I$, set 
$f_{ij}:={\alpha_j^{-1}}|_{U_{ij}}\circ {\alpha_i}|_{U_{ij}}$. Then 
$$
at(\sQ)=\big[(-{\alpha_j}|_{U_{ij}}\circ {df_{ij}}|_{U_{ij}}\circ {\alpha_i^{-1}}|_{U_{ij}})_{i,j}\big]\in 
H^1(X,\Omega^1_X\otimes\sE\hspace{-0.07cm}\textit{nd}_{\sO_X}(\sQ)).
$$
Since $\alpha_i=\alpha_j \cdot f_{ij}$ on $U_{ij}$,
$$d\alpha_i=d\alpha_j \cdot f_{ij}+\alpha_j \wedge df_{ij},$$
and
$$\alpha_i \wedge \beta_i=\alpha_j \cdot f_{ij}\wedge \beta_i
=\alpha_j \wedge \beta_j \cdot f_{ij}+\alpha_j \wedge df_{ij}.$$
Given $\vec v\in H^0(U_{ij},\sF_{|U_{ij}})$,

\begin{eqnarray*}
\alpha_j \cdot f_{ij}\cdot \beta_i(\vec v) & = & i_{\vec v}(\alpha_j \cdot f_{ij}\wedge \beta_i)\\
& = & i_{\vec v}(\alpha_j \wedge \beta_j \cdot f_{ij}+\alpha_j \wedge df_{ij}).\\
& = & \alpha_j \cdot \beta_j (\vec v) \cdot f_{ij}+\alpha_j \cdot df_{ij}(\vec v).
\end{eqnarray*}
This implies that
$$
\delta(at(\sQ))
=
\big[({\beta_j}|_{U_{ij}}-{\beta_i}|_{U_{ij}})_{i,j}\big],
$$
and thus
$$\delta(at(\sQ))=0.$$
\end{proof}

The following  result is certainly well known. We include a proof for lack of adequate reference.
We remark that a locally free sheaf of finite rank on a Cohen-Macaulay locally 
noetherian
scheme satisfies Serre's condition
$S_k$ for any integer $k\ge 0$.

\begin{lemma}\label{lemma:restriction_cohomology}
Let $X$ be a noetherian scheme, and $\sG$ a coherent sheaf of $\sO_X$-modules on $X$. Suppose that $\sG$ 
satisfies Serre's condition $S_k$ 
for some integer $k\ge 1$. Let $U\subset X$ be an open subset such that $codim_X(X\setminus U)\ge k$. Then 
$H^i(X,\sG)\simeq H^i(U,\sG_{|U})$ for $0\le i\le k-1$.
\end{lemma}

\begin{proof}Set $Y:=X\setminus U$. Denote by $\sH_Y^i(\sG)$ the local cohomology sheaf. By assumption, $\sH_Y^i(\sG)=0$ for
$0\le i\le k$. The spectral sequence of local cohomology groups and sheaves implies that $H_Y^i(X,\sG)=0$.
The long exact sequence
$$\cdots \to H_Y^{i}(X,\sG) \to H^i(X,\sG) \to H^i(U,\sG_{|U}) \to H_Y^{i+1}(X,\sG) \to \cdots$$
then yields $H^i(X,\sG)\simeq H^i(U,\sG_{|U})$ for $0\le i\le k-1$, as claimed.
\end{proof}

We can now prove our result in the mildly singular setting.

\begin{proof}[{Proof of Theorem \ref{thm:regular_fano_fols}}]
Let $\sF\subsetneq T_X$ be a codimension 1 
foliation on a klt projective variety $X$, and suppose that 
$X$ and $\sF$ are both regular in codimension 2.
We want to show that  $-K_\sF$ is not ample.

Assume to the contrary that $-K_\sF$ is ample.
Set $\sQ:=T_X/\sF$ and $\sL:=\det(\sF)$. Then $\sQ^{**}\simeq \sO_X(-K_X)\otimes \sL^{\otimes -1}$.
Let $U\subset X$ be a smooth open subset such that $\codim(X\setminus U)\ge 3$, and
$\sF$ is regular on $U$. 
By 
Lemma \ref{lemma:restriction_cohomology}, 
$h^1(U,\sQ^*_{|U})=h^1(X,\sQ^*)=h^1(X,\sO_X(K_X)\otimes \sL)$, 
and the latter vanishes by the Kawamata-Viehweg vanishing theorem (see \cite[Theorem 1.2.5]{kmm}). 
By Lemma \ref{lemma:atiyah_class_foliation}, 
$0=at(\sQ_{|U})\in 
H^1(U,\Omega_U^1)$. This implies that, for any smooth complete curve
$C \subset U$, $(K_X+c_1(\sL))\cdot C=0$. Hence 
$K_X+c_1(\sL)\equiv 0$. Thus, $X$ is a $\bQ$-Fano variety.
Let $k$ be a positive integer such that $k(K_X+c_1(\sL))$ is Cartier. 
By Lemma \ref{lemma:Fano_lin_equ_and_num_equ}, $\sO_X(k(K_\sF+c_1(\sL)))\simeq \sO_X$.
This contradicts Lemma \ref{lemma:vanishing_weak_fano} since
$h^0\big(X,\Omega_X^{q}[\otimes]\sO_X(-K_\sF-c_1(\sL))\big)\neq0$.
\end{proof}

\bibliographystyle{amsalpha}
\bibliography{foliation}

\end{document}